\documentclass[2pt,a4paper]{article}
\usepackage{amsmath}
\usepackage{amsfonts}
\usepackage{amsthm}
\usepackage{amssymb}
\usepackage{booktabs,amsmath}

\def\XXint#1#2#3{{\setbox0=\hbox{$#1{#2#3}{\int}$}
     \vcenter{\hbox{$#2#3$}}\kern-.5\wd0}}

\usepackage[english]{babel}
\usepackage{xcolor}
\usepackage{enumerate}

\theoremstyle{plain}
\newtheorem{theo}{Theorem}[section]
\newtheorem{lem}[theo]{Lemma}
\newtheorem{prop}[theo]{Proposition}
\newtheorem{cor}[theo]{Corollary}%

\theoremstyle{definition}
\newtheorem{definition}[theo]{Definition}
\theoremstyle{remark}
\newtheorem{rem}[theo]{Remark}
\usepackage{tabularx}

\usepackage{amsmath,amsfonts,amssymb,amscd,lscape,amstext,amsthm}

\numberwithin{equation}{section}

\newcommand{\R}{\mathbb{R}}
\newcommand{\N}{\mathbb{N}}
\newcommand{\M}{\mathbb{M}}
\newcommand{\K}{\mathbb{K}}

\title{Strong unique continuation and global regularity\\
 estimates for nanoplates

}

\author{Antonino Morassi\thanks{Dipartimento Politecnico di Ingegneria e Architettura,
Universit\`a degli Studi di Udine, via Cotonificio 114, 33100
Udine, Italy. E-mail: \textsf{antonino.morassi@uniud.it}}, \  Edi
Rosset\thanks{Dipartimento di Matematica e Geoscienze,
Universit\`a degli Studi di Trieste, via Valerio 12/1, 34127
Trieste, Italy. E-mail: \textsf{rossedi@units.it}}, \ Eva 
Sincich\thanks{Dipartimento di Matematica e Geoscienze,
	Universit\`a degli Studi di Trieste, via Valerio 12/1, 34127
	Trieste, Italy. E-mail: \textsf{esincich@units.it}} \
and Sergio
Vessella\thanks{Dipartimento di Matematica e Informatica ``Ulisse
Dini'', Universit\`a degli Studi di Firenze, Via Morgagni 67/a,
50134 Firenze, Italy. E-mail: \textsf{sergio.vessella@unifi.it}}}

\date{}

\begin{document}

\maketitle

\noindent \textbf{Abstract:} 
In this paper we analyze some properties of a sixth order elliptic operator arising in the framework of the strain gradient linear elasticity theory for nanoplates in flexural deformation. We first rigorously deduce the weak formulation of the underlying Neumann problem as well as its well posedness. Under some suitable smoothness assumptions on the coefficients and on the geometry we derive interior and boundary regularity estimates for the solution of the Neumann problem. Finally, for the case of isotropic materials, we obtain new Strong Unique Continuation results in the interior, in the form of doubling inequality and three spheres inequality by a Carleman estimates approach. 

\medskip

\medskip
 

\medskip

\medskip



\section{Introduction} 
\label{sec:introduction}

In this work, we begin a line of research aimed at studying some recent models for two-dimensional micro and nanomechanical systems, which we will refer to as \textit{nanoplates}. In particular, here we deal with the formulation and well-posedness of the direct problem describing the static equilibrium of a nanoplate under Neumann boundary conditions, and we derive some unique continuation properties of the solutions to the equation of nanoplates in bending. The latter properties, as is well known, constitute the essential tool for the study of inverse boundary value problems.

Nanoplates are nowadays widespread as mass sensors, biomarkers or gas sensors, as well as actuators for vibration control purposes \cite{Eom-PR-2011,HJIAY-IJNLM-2020}. The plate typology, although less common than nanobeams, has some inherent mechanical advantages that include robustness, which is a relevant feature for fabrication and functionalisation, and higher stiffness, which results in higher frequencies and small free vibration energy dissipation in both fluid and gaseous environments \cite{Shen-CMS-2012,Bhaswara-JMM-2014}. Albeit the detection of added mass is one of the most popular issues in applications \cite{HKCMSR-NN-2015,MYW-APE-2019,DFFSMZ-IJSS-2020}, other notable  inverse problems for nanoplates involve force or pressure sensing {}from dynamic data \cite{KMZ-MMAS-2022}. In addition, there has recently been growing interest in the development of diagnostic techniques for assessing the presence of defects in nanoplates, thus paving the way for the extension of methods hitherto designed for large-scale mechanical systems to the nanodimensional size as well \cite{YZC-SV-2020}.

The modeling of nanoplates presents specific requirements due to the presence of size effects and, therefore, classical Continuum Mechanics, as a length-scale free theory, loses its predictive capacity in this field. In the last two decades various Generalized Continuum Mechanics theories, such as Couple Stress, Nonlocal or Strain/Stress Gradient, have been proposed to model nanostructures and, specifically, nanoplates within the linear elasticity setting. Among these theories, the Simplified Strain Gradient Elasticity theory (SSGE) developed by Lam et al. \cite{Lam2003} has achieved a remarkable diffusion and has been applied to a wide variety of one-dimensional nanostructures, see, among others contributions, \cite{AC-JVC-2014, KZNW-IJES-2009}. Some recent works address the study of Kirchhoff-Love's nanoplates using SSGT \cite{MM-EIMS-2013,WZZC-EIMS-2011}. Let us begin by recalling the partial differential equation that expresses the static equilibrium of a nanoplate in bending under vanishing body forces and couples. Let $\Omega$ be the middle surface of a nanoplate having uniform thickness $t$, and let us denote by $(x_1,x_2)$ the Cartesian coordinates of a point of $\Omega$. It turns out that the transverse displacement $u=u(x_1,x_2)$ of the nanoplate satisfies the sixth order elliptic equation
\begin{equation}
	\label{eq:introd-6th-PDE}
	\frac{\partial^2 }{\partial x_i \partial x_j}\left (
	-(P_{ijlm}+P_{ijlm}^h   ) \frac{\partial^2 u}{\partial x_l \partial x_m}
	+
	\frac{\partial }{\partial x_k} \left ( Q_{ijklmn}\frac{\partial^3 u}{\partial x_l \partial x_m \partial x_n} \right )
	\right )=0	
	\quad \hbox{in } \Omega,
\end{equation}
where the summation over repeated indexes $i,j,k,l,m,n=1,2$ is assumed. Here, $P_{ijlm}$ are the Cartesian components of the fourth-order tensor describing the material response in classical Kirchhoff-Love theory, whereas $P_{ijlm}^h$ and $Q_{ijklmn}$ are the components of a fourth- and sixth-order tensor, respectively, that account for the three material length scale parameters in the SSGT. We refer to Section \ref{sec:direct} for explicit expressions of the above tensors for an isotropic material. 

The mechanical model of the nanoplate is completed by specifying the conditions that apply at the boundary of $\Omega$. As it occurs in the classical Kirchhoff-Love's theory of plates \cite{K-JRAM-1850}, one of the subtle issues is the determination of the Neumann boundary conditions \cite{WHZZ-AMM-2016}. It should be noted that these conditions also play a crucial role in the formulation of inverse boundary value problems, as they determine the correlation between assigned quantities and unknown quantities at the boundary. Neumann boundary conditions on a curved smooth boundary were derived by Papargyri-Beskou \cite{PGB-IJSS-2010} for Kirchhoff-Love nanoplates within a simplified version of Mindlin's theory with a single scale constant \cite{MindEsh1968}. Unfortunately, the transformation {}from the fixed system of Cartesian coordinates to the local system of coordinates on the boundary used in \cite{PGB-IJSS-2010} did not take into account of the possible variation of the local natural basis $( {n }, {\tau}  )$ in case of curved boundary. Here, the vectors ${n}$, ${\tau}$ are the unit outer normal and the unit tangent to the boundary, respectively. The oversight was present also in the treatment by Lazopoulos \cite{L-MRC-2009}, and in \cite{NN-EIMS-2017} only natural conditions on straight portions of the boundary were considered. 

In what follows we refer to the mechanical Kirchhoff-Love's nanoplate model proposed in \cite{KMZ-MMAS-2022} within the SSGT elasticity, in which the Neumann conditions are correctly derived for smooth boundary (see problem \eqref{eq:M2-1}--\eqref{eq:M2-4}). In Section $3$ we propose an alternative, although equivalent, determination of these boundary conditions (Lemma \ref{Lemma:Fichera}) and we develop the variational formulation of the Neumann problem. We prove existence and uniqueness of the solution (Proposition \ref{prop:DirectProblem}), and improved regularity at the interior (Theorem \ref{theo:loc-int-reg-res-improved}) and up to the boundary (Theorem \ref{prop:glo-reg-res-exProp8-2}). These properties are derived within the framework of high-order boundary value elliptic problems in variational form, and are instrumental to the second contribution of this paper, namely the quantitative unique continuation results obtained in Section $4$ and which we describe below.

First of all, let us recall some basic notions concerning the Unique Continuation Properties. We say that a linear partial differential equation 
\begin{equation}\label{intro:UC}
	L(u) =0 \quad \mbox{in } B,
\end{equation}
 where $B\subset \mathbb{R}^n$ is an open ball (or, more generally, is a connected open set) enjoys the Weak Unique Continuation Property  (WUCP) if the following property holds true: for any open subset $\omega$ of $B$,
$$L(u) = 0 \quad \mbox{in } B \quad \mbox{ and }  \quad u \equiv 0 \quad \mbox{ in } \omega,$$ 
imply $u \equiv 0$ in $B$. 

We say that $L$  enjoys the Strong Unique Continuation Property (SUCP) if the following property holds true: for any point $x_0\in B$ and for any solution $u$ to \eqref{intro:UC} which satisfies

$$\int_{B_{\varrho}(x_0)}u^2=\mathcal{O}\left(\varrho^N\right),\quad \mbox{as } \varrho\rightarrow 0,\  \forall N\in \mathbb{N},$$ 
it follows that
$$u\equiv 0,\quad\mbox{in } B.$$
It is obvious that SUCP implies WUCP.

In the present paper we are interested in quantitative versions of the SUCP. More precisely, we are interested in doubling inequality, which typically takes the form
\begin{equation}
	\label{eq:doubling-w-intro}
    \int_{B_{2r}(x_0)}u^2\leq C \int_{B_{r}(x_0)}u^2,
\end{equation}
where $C$ depends on $u$ but not on $r$. As shown in \cite{Ga-Li}, a consequence of doubling inequality is the SUCP for solutions to \eqref{intro:UC}.

Here, we will prove the doubling inequality for the sixth order equation of nanoplates \eqref{eq:introd-6th-PDE} consisting of inhomogeneous isotropic material (see Section \ref{sec:Doubl-Three-Sphere}). As recalled above, the doubling inequality implies SUCP. A crucial point to derive the doubling inequality for the nanoplate is to observe that the solution to equation \eqref{eq:introd-6th-PDE}, with the tensors $P_{ijlm}$, $P_{ijlm}^h$ and $Q_{ijklmn}$ given respectively by  \eqref{eq:M3-1}, \eqref{eq:M3-2} and \eqref{eq:M4-4}, satisfies the following differential inequality
\begin{equation}\label{diffineq-intro}
|\Delta^3 u|\le M \left(|D\Delta ^2 u|  + \sum_{k=0}^4 |D^k u|\right),
\end{equation}
where $M$ is a positive constant. Now, \eqref{diffineq-intro} tells us that our task of proving the doubling inequality is certainly included in the more general issue of Unique Continuation Property for equations whose principal part is a power of Laplace operator. The literature on this subject is very extensive, here we only remember some papers; namely those concerning the WUCP \cite{Goo},  \cite{LR-Ro}, \cite{Ni}, \cite{Pe}, \cite{Pr}, \cite{Wat} and those concerning  the SUCP \cite{AMRV2009},   \cite{Co-Gr}, \cite{Co-Ko}, \cite{DLMRVW}, \cite{LiNW}, \cite{MRV2007},  \cite{Zhu}.


 
 A good part of the interest on the subject of WUCP for equations whose principal part is $\Delta^m$ arises {}from the fact that the operator $\Delta^m$ has multiple complex characteristic for $m\geq2$, $m\in \mathbb{N}$, hence the general theory of Carleman estimates, conceived by H\"{o}rmander, \cite{HO63}, see also \cite{Ler}, does not apply. Here we recall that Carleman estimates were introduced in 1939 in \cite{Ca} by the homonymous  mathematician, to prove the uniqueness for a Cauchy problem for elliptic systems in two variables with nonanalytic coefficients. Carleman estimates are today the most powerful and general tool to study the Unique Continuation Property of PDEs.
 
 As it has been proven in \cite{Pr}, the Unique Continuation Property for solutions to the inequality
\begin{equation}\label{diffineq-intro-1}
|\Delta^m u|\le M \sum_{k=0}^h |D^k u|,
\end{equation}
holds true whenever $h=[\frac{3m}{2}]$ (for every $p\in \mathbb{R}$, $[p]$ represents the integer part of $p$). In particular, in our case we have $m=3$, hence $h=4$. On the other side it was proven, see \cite[Ch. 2, Sect. 3]{Zu} and \cite[Sect. 5]{Goo} and references therein, that there exists an operator of the form
$$\Delta^3+L_5,$$
where $L_5$ is a fifth-order operator with continuous coefficients (of complex values), for which the unique continuation property fails. At the light of previous considerations, the case \eqref{diffineq-intro} is, in some sense, an ``intermediate case'' for which the WUCP  and even more, the SUCP is worth to study.

Our main results are : 

(a) a Carleman estimates for the cube Laplacian contained in Proposition \ref {prop:Carleman 3Laplace} which has been obtained by a careful iteration of the Carleman estimate with a suitable singular weight for Laplace operator, see Proposition \ref{prop:Carlm-delta}) (see also  \cite{mrv:mat-cat}). 

(b) the derivation of a doubling inequality for solutions $u$ to \eqref{diffineq-intro} in the form of \eqref{eq:doubling-w-intro}, see Corollary  \ref{SUCP} for precise statement.

(c) three sphere inequalities derived by the above mentioned doubling inequalities that as widely illustrated in \cite{A-R-R-V}, are fundamental tools to obtain estimates of propagation of smallness.

Quantitative estimates in the form of doubling inequalities and of three spheres inequalities have shown to be extremely
useful in the treatment of inverse boundary value problems associated to
the fourth order elliptic equation of the classical Kirchhoff-Love plate in bending \cite{MRV2007}. In a subsequent paper \cite{MRSV-2022-SE} we plan to apply such estimates to diagnostic problems of non-destructive testing for elastic nanoplates, which are modeled as inverse boundary value problems of determining, within a nanoplate $\Omega$, the possible presence of an inclusion made of different material {}from boundary measurements of Neumann data and corresponding work-conjugate quantities.  As a matter of fact, the doubling inequalities and their connection with $A_p$ weight, lead the way to estimates of measure (area) of very general unknown inclusions. We refer to \cite{AMR2003} for the context of second order elliptic equations and systems and to \cite{AMRV2009}, \cite{DLMRVW}, \cite{MRV2009} for the context of fourth order plate equation.

The plan of the paper is as follows. In Section \ref{sec:notation} we collect some notation and definitions. In Section \ref{sec:direct} we introduce the formulation of the direct problem for the nanoplate mechanical model and we prove its well-posedness (Proposition \ref{prop:DirectProblem}). We further analyze the properties of the solution to the nanoplate equilibrium problem  by providing a global regularity result (Theorem \ref{prop:glo-reg-res-exProp8-2}) and an improved regularity result  in the interior (Theorem \ref{theo:loc-int-reg-res-improved}). The unique continuation issues are contained in Section \ref{DTSI}, where we first derive a Carleman estimate for the cube Laplacian (Proposition \ref{prop:Carleman 3Laplace}). By such an estimates  we achieve a doubling inequality for the solution (Theorem \ref{theo:40.teo}) and a refined version of it which allows us to deduce a three Spheres inequality for the solution at hand (Corollary \ref{SUCP}) . In the Appendix, we perform a change of variable argument to express the second order derivative on boundary points in terms of intrinsic coordinates (proof of Lemma \ref{Lemma:Fichera}). 
\section{Notation} 
\label{sec:notation}
\noindent 
Let $P=(x_1(P), x_2(P))$ be a point of $\R^2$. We shall denote by
$B_r(P)$ the disk in $\R^2$ of radius $r$ and center $P$ and by
$R_{a,b}(P)$ the rectangle of center $P$ and sides parallel to the
coordinate axes, of length $2a$ and $2b$, namely
\begin{equation}
\label{eq:rectangle}
R_{a,b}(P)=\{x=(x_1,x_2)\ |\ |x_1-x_1(P)|<a,\ |x_2-x_2(P)|<b \}.     
\end{equation}
\begin{definition}
	\label{def:reg_bordo} (${C}^{k,\alpha}$ regularity)
	Let $\Omega$ be a bounded domain in ${\R}^{2}$. Given $k,\alpha$,
	with $k\in\N$, $k\geq 1$, $0<\alpha\leq 1$, we say that a portion $S$ of
	$\partial \Omega$ is of \textit{class ${C}^{k,\alpha}$ with
		constants $r_{0}$, $M_{0}>0$}, if, for any $P \in S$, there exists
	a rigid transformation of coordinates under which we have $P=0$
	and
	\begin{equation*}
	\Omega \cap R_{r_0,2M_0r_0}=\{x \in R_{r_0,2M_0r_0} \quad | \quad
	x_{2}>g(x_1)
	\},
	\end{equation*}
	where $g$ is a ${C}^{k,\alpha}$ function on $[-r_0,r_0]$
	satisfying
	\begin{equation*}
	g(0)=g'(0)=0,\quad |g\|_{{C}^{k,\alpha}([-r_0,r_0])} \leq M_0r_0,
	\end{equation*}
	where
	\begin{equation*}
	\|g\|_{{C}^{k,\alpha}([-r_0,r_0])} = \sum_{i=0}^k
	r_0^i\sup_{[-r_0,r_0]}|g^{(i)}|+r_0^{k+\alpha}|g|_{k,\alpha},
	\end{equation*}
	
	\begin{equation*}
	|g|_{k,\alpha}= \sup_ {\overset{\scriptstyle t,s\in
			[-r_0,r_0]}{\scriptstyle t\neq s}}\frac{|g^{(k)}(t) -
		g^{(k)}(s)|}{|t-s|^\alpha}.
	\end{equation*}
\end{definition}
We use the convention to normalize all norms in such a way that
their terms are dimensionally homogeneous and coincide with the
standard definition when the dimensional parameter equals one. For
instance, given a function $u:\Omega \rightarrow \R$ we denote 
\begin{equation}
\label{eq:1.1.1}
\|u\|_{H^k(\Omega)}=r_0^{-1} \left ( \sum_{i=0}^k r_0^{2i}\int_\Omega |D^i u|^2 
 \right )^{\frac{1}{2}} \hbox{with}  \int_{\Omega} |D^k u |^2=\int_{\Omega}\sum_{|\alpha|=k} |D^\alpha u |^2
\end{equation}
and so on for boundary and trace norms.
For any $h>0$ we set 
$$\Omega_{h}=\{x\in \Omega \ : \ \mbox{dist}(x,\partial \Omega)>h \}.$$
Given a bounded domain $\Omega$ in $\R^2$ such that $\partial
\Omega$ is of class $C^{k,\alpha}$, we consider as
positive the orientation of the boundary induced by the outer unit
normal $n$ in the following sense. Given a point
$P\in\partial\Omega$, let us denote by $\tau=\tau(P)$ the unit
tangent at the boundary in $P$ obtained by applying to $n$ a
counterclockwise rotation of angle $\frac{\pi}{2}$, that is $\tau=e_3 \times n$, 
where $\times$ denotes the vector product in $\R^3$ and $\{e_1,
e_2, e_3\}$ is the canonical basis in $\R^3$.

Given any connected component $\cal C$ of $\partial \Omega$ and
fixed a point $P_0\in\cal C$, let us define as positive the
orientation of $\cal C$ associated to an arclength
parameterization $\psi(s)=(x_1(s), x_2(s))$, $s \in [0, l(\cal
C)]$, such that $\psi(0)=P_0$ and $\psi'(s)=\tau(\psi(s))$. Here
$l(\cal C)$ denotes the length of $\cal C$.

Throughout the paper, we denote by $w,_\alpha$, $\alpha=1,2$,
$w,_s$, and $w,_n$ the derivatives of a function $w$ with respect
to the $x_\alpha$ variable, to the arclength $s$ and to the normal
direction $n$, respectively, and similarly for higher order
derivatives.

We denote by $\M^{2}, \M^{3}$ the Banach spaces of second order and the third order tensors and by 
$\widehat{\M}^{2}, \widehat{\M}^{3}$ the corresponding subspaces of tensors having components invariant with respect to permutations of the indexes.

Let ${\cal L} (X, Y)$ be the space of bounded linear
operators between Banach spaces $X$ and $Y$. Given $\K\in{\cal L} ({\M}^{2},{\M}^{2})$ and $A,B\in \M^{2}$, we use the following notation 
\begin{equation}
\label{eq:2.notation_1}
({\K}A)_{ij} = \sum_{l,m=1}^{2} K_{ijlm}A_{lm}, \quad A \cdot B = \sum_{i,j=1}^{2} A_{ij}B_{ij} \ .
\end{equation}
Similarly, given $\K\in{\cal L} ({\M}^{3},
{\M}^{3})$ and $A,B\in \M^{3}$, we denote 
\begin{equation}
\label{eq:2.notation_1bis}
({\K}A)_{ijk} = \sum_{l,m,n=1}^{2} K_{ijklmn}A_{lmn},\quad A \cdot B = \sum_{i,j,k=1}^{2} A_{ijk}B_{ijk} .
\end{equation}
Moroever, for any $A\in \M^{n}$ with $ n=2,3$ we shall denote 
\begin{equation}
\label{eq:2notation_3}
|A|= (A \cdot A)^{\frac {1} {2}}.
\end{equation}
The linear space of the infinitesimal rigid displacements is defined as
\begin{equation}
\label{eq:def_rig_displ-1}
{\cal R}_2 = \left \{
r(x) = c + Wx, \ c \in \R^2, \ W \in \M^2, \ W+W^T=0
\right \}.
\end{equation}
%

\noindent Throughout the paper, summation over repeated indexes is assumed.

\section{The Neumann problem} 
\label{sec:direct}

\subsection{Nanoplate mechanical model}
\label{subsec:a priori}
Let us consider a nanoplate $\Omega \times  \left ( -\frac{t}{2}, \frac{t}{2}   \right )$, where the middle surface $\Omega$ is a bounded domain of $\R^2$, and the thickness $t$ is constant and small with respect to $\mbox{diam}(\Omega)$, i.e. $t << \mbox{diam} (\Omega)$. We assume that the boundary $\partial \Omega$ of $\Omega$ is of class $C^{2,1}$, with constants $r_0$, $M_0$. Moreover, given $M_1>0$, 
\begin{equation}
	\label{eq:M1-1}
	|\Omega| \leq M_1 r_0^2 .
\end{equation}
The material of the nanoplate is assumed to be linearly elastic, inhomogeneous, center-symmetric and isotropic, according to the simplified version of Toupin \cite{Tou1962}, \cite{Tou1964} and Mindlin and Eshel \cite{MindEsh1968} theories proposed by Lam et al. in \cite{Lam2003}.

Under the kinematic framework of the Kirchhoff-Love theory, and for infinitesimal deformation, the statical equilibrium problem of the nanoplate loaded at the boundary and under vanishing body forces is described by the following Neumann boundary value problem \cite{KMZ-MMAS-2022}:
\begin{equation}
	\label{eq:M2-1}
		(M_{\alpha \beta} +  \overline{M}_{\alpha \beta \gamma, \gamma}^h   )_{,\alpha \beta}=0 \quad \hbox{in } \Omega,
\end{equation}
\begin{multline}
	\label{eq:M2-2}
	(
	M_{\alpha \beta} +  \overline{M}_{\alpha \beta \gamma, \gamma}^h
	)_{,\alpha} n_\beta
	+
	(
	(
	M_{\alpha \beta} +  \overline{M}_{\alpha \beta \gamma, \gamma}^h
	)n_\alpha {\tau}_\beta	
	)_{,s}
	+
	(
	\overline{M}_{\alpha \beta \gamma}^h {\tau}_\alpha {\tau}_\beta n_\gamma
	)_{,ss} -
	\\
	-
	(
	\overline{M}_{\alpha \beta \gamma}^h n_\gamma
	(  {\tau}_{\alpha,s} {\tau}_\beta - n_{\alpha,s} n_\beta )
	)	
	_{,s}
	=
	- \widehat{V}  \qquad \hbox{on } \partial \Omega,
\end{multline}
\begin{multline}
	\label{eq:M2-3}
	(
	M_{\alpha \beta} +  \overline{M}_{\alpha \beta \gamma, \gamma}^h
	)n_\alpha n_\beta 
	+
	(
	\overline{M}_{\alpha \beta \gamma}^h n_\gamma ({\tau}_\alpha n_\beta + {\tau}_\beta n_\alpha )
	)_{,s} - \overline{M}_{\alpha \beta \gamma}^h n_\gamma
	(n_{\alpha,s} {\tau}_\beta)
	=\widehat{M}_n  \qquad \hbox{on } \partial \Omega,
\end{multline}
\begin{equation}
	\label{eq:M2-4}
	  \quad\quad \quad \quad  \quad \quad \quad \quad \quad \quad \quad \quad \quad \quad \quad \quad \quad \quad \quad \quad \quad \quad \quad \quad  \quad \quad \quad \quad  \quad  \overline{M}_{\alpha \beta \gamma}^h n_\alpha n_\beta n_\gamma
	=-\widehat{M}_n^h  \qquad \hbox{on } \partial \Omega.
\end{equation}
The functions $M_{\alpha \beta}= M_{\alpha \beta}(u)$, $\overline{M}_{\alpha \beta \gamma}^h=\overline{M}_{\alpha \beta \gamma}^h(u)$, $\alpha, \beta, \gamma=1,2$, in the above equations are the Cartesian components of the couple tensor $M=(M_{\alpha \beta})$ and the high-order couple tensor $\overline{M}^h=(\overline{M}_{\alpha \beta \gamma}^h)$, respectively, corresponding to the transverse displacement $u(x_1,x_2)$, $u : \Omega \rightarrow \R$, of the point $(x_1,x_2)=x$ belonging to the middle surface of the nanoplate. To simplify the notation, the dependence of these quantities on $u$ is not explicitly indicated in \eqref{eq:M2-1}--\eqref{eq:M2-4} and in what follows.

As shown in \cite{KMZ-MMAS-2022}, the functions $M_{\alpha\beta}$ are given by
\begin{equation}
	\label{eq:M2-5}
	M_{\alpha\beta} = -( P_{\alpha \beta \gamma \delta} + P_{\alpha \beta \gamma \delta}^h  ) u_{,\gamma \delta},
\end{equation}
where the fourth order tensors $\mathbb{P} = \mathbb{P}(x) \in L^\infty (\Omega,  \mathcal{L}(\widehat{\M}^2, \widehat{\M}^2) )   $, $\mathbb{P}^h = \mathbb{P}^h(x) \in L^\infty (\Omega,  \mathcal{L}( \widehat{\M}^2, \widehat{\M}^2)   )   $ have Cartesian components $P_{\alpha \beta \gamma \delta}$, $P_{\alpha \beta \gamma \delta}^h $ given by
\begin{equation}
	\label{eq:M3-1}
   	P_{\alpha \beta \gamma \delta}= B((1-\nu)\delta_{\alpha \gamma} \delta_{\beta \delta} + \nu \delta_{\alpha \beta} \delta_{\gamma\delta}
   	),
\end{equation}
\begin{equation}
	\label{eq:M3-2}
	P_{\alpha \beta \gamma \delta}^h= (2a_2+5a_1)\delta_{\alpha \gamma} \delta_{\beta \delta} + (-a_1 -a_2 +a_0) \delta_{\alpha \beta} \delta_{\gamma\delta}.
\end{equation}
It is easy to verify that, for every $A,B\in \widehat{\M}^2$,
\begin{equation}
	\label{eq:M3-3}
	\mathbb{P}A\cdot B=\mathbb{P}B\cdot A,\quad  \mathbb{P}^h A\cdot B=\mathbb{P}^h B\cdot A, \quad  \hbox{a.e. in } \ \Omega,
\end{equation}
The bending stiffness (per unit length) $B=B(x)$ is given by the function
\begin{equation}
	\label{eq:M3-5}
	B(x) = \frac{t^3 E(x)}{12(1-\nu^2(x))},
\end{equation}
where the Young's modulus $E$ and the Poisson's coefficient $\nu$ of the material can be written in terms of the Lamé moduli $\mu$ and $\lambda$ as follows
\begin{equation}
	\label{eq:M3-6}
	E(x) = \frac{\mu(x)(2\mu(x)+3\lambda(x))}{ \mu(x) +\lambda(x)  },
	\quad 
	\nu(x) = \frac{\lambda(x)}{2(\mu(x)+\lambda(x))}.
\end{equation}
On $\mu$ and $\lambda$ we assume the following ellipticity conditions:
\begin{equation}
	\label{eq:M3-7}
\mu(x) \geq \alpha_0 >0, \quad 2\mu(x)+3\lambda(x) \geq \gamma_0 >0 \quad \hbox{a.e. in } \ \Omega,
\end{equation}
where $\alpha_0$, $\gamma_0$ are positive constants.

The coefficients $a_i(x)$, $i=0,1,2$, are given by (see \cite{KMZ-MMAS-2022})
\begin{equation}
	\label{eq:M3-8}
	a_0(x)=2\mu(x)t\mathit{l}_0^2, \quad a_1(x) = \frac{2}{15}\mu(x)t  \mathit{l}_1^2, \quad a_2(x) = \mu(x)t\mathit{l}_2^2, 
\end{equation}
where the material length scale parameters $\mathit{l}_i$ are assumed to be positive constants. 
Denoting
\begin{equation}
	\label{eq:M4-1}
	l=\min\{l_0,l_1,l_2\},
\end{equation}
by \eqref{eq:M3-7}--\eqref{eq:M4-1}, we have
\begin{equation}
	\label{eq:M4-2}
	a_i(x) \geq tl^2\alpha_0^h > 0, \quad i=0,1,2, \quad \hbox{a.e. in } \Omega,
\end{equation}
where $\alpha_0^h = \frac{2}{15} \alpha_0$. 

The functions $\overline{M}_{ijk}^h$ ($i,j,k=1,2$) are given by
\begin{equation}
	\label{eq:M4-3}
	\overline{M}_{ijk}^h= Q_{ijklmn}u_{,lmn},
\end{equation}
where the sixth order tensor $\mathbb{Q}=\mathbb{Q}(x) \in L^\infty( \Omega,   \mathcal{L}(\widehat{\M}^3, \widehat{\M}^3   )    )$ can be expressed as follows ($i,j,k,l,m,n=1,2$)
\begin{multline}
	\label{eq:M4-4}
	Q_{ijklmn}= \frac{1}{3}(b_0 -3b_1)\delta_{ij}\delta_{kn}\delta_{lm}+
	\\
	+ \frac{1}{6}(b_0 -3b_1)
	(
	\delta_{ik} ( \delta_{jl} \delta_{mn} + \delta_{jm}\delta_{ln}   )
	+
	\delta_{jk} ( \delta_{il} \delta_{mn} + \delta_{im}\delta_{ln}   )
	)
	+
	Q_8 
	(
	\delta_{kn} ( \delta_{il}\delta_{jm} +\delta_{im}\delta_{jl}    )
	)
	+
	\\
	+
	Q_9
	(
	\delta_{jn} ( \delta_{il} \delta_{km} + \delta_{im}\delta_{kl}   )
	+
	\delta_{in} ( \delta_{jl} \delta_{km} + \delta_{jm}\delta_{kl}   )
	),	
\end{multline}
where 
\begin{equation}
	\label{eq:M4-5}
	2(Q_8+2Q_9)=5b_1 \ , 
\end{equation} 
\begin{equation}
	\label{eq:M5-1}
	b_0(x)=2\mu(x)\frac{t^3}{12}\mathit{l}_0^2, \quad 
	b_1(x)=\frac{2}{5}\mu(x)\frac{t^3}{12}\mathit{l}_1^2
	\quad \hbox{a.e. in } \ \Omega.
\end{equation}
Note that, by \eqref{eq:M4-4}--\eqref{eq:M4-5}, the constitutive equations \eqref{eq:M4-3} become
\begin{equation}
	\label{eq:M5-2}
	\overline{M}_{ijk}^h= \frac{1}{3}(b_0 -3b_1)
	(   
	\delta_{ij} u_{, mm k} + \delta_{ik} u_{, mm j}
	+ \delta_{jk} u_{, mm i}	
	) + 5b_1 u_{,ijk}.
\end{equation}
It is easy to verify that, for every $A,B\in \widehat{\M}^3$,
\begin{equation}
	\label{eq:M5-3}
	\mathbb{Q}A\cdot B = \mathbb{Q}B\cdot A, \quad \hbox{a.e. in } \ \Omega.
\end{equation}
The functions $\widehat{V}$ (shear force), $\widehat{M}_n$ (bending moment) and $\widehat{M}_n^h$ (high-order bending moment) appearing in the equilibrium boundary equations \eqref{eq:M2-2}--\eqref{eq:M2-4} are the work conjugate actions to the deflection $u$, to the normal derivative $ \frac{\partial u}{\partial n}$ and to the second normal derivative  $ \frac{\partial^2 u}{\partial n^2}$ at the boundary $\partial \Omega$, respectively. On these quantities we require the following regularity conditions
\begin{equation}
	\label{eq:M5-4}
	\widehat{V} \in H^{ - 5/2  }(\partial \Omega), \quad 
	\widehat{M}_n \in H^{ - 3/2  }(\partial \Omega), \quad 
	\widehat{M}_n^h \in H^{ - 1/2  }(\partial \Omega).
\end{equation}
In order to simplify our notation, throughout the paper we will denote by $C, C_1, C_2,\dots$ positive constants which may vary {}from line to line.

\subsection{Variational formulation and well-posedness of the Neumann problem}
\label{subsec:variational}

In view of the variational formulation of the equilibrium problem \eqref{eq:M2-1}--\eqref{eq:M2-4}, we recall the following ellipticity result.

\begin{lem}[Strong convexity of the strain energy density; \cite{KMZ-MMAS-2022}]
	\label{Lemma:M6-1}
Let the tensors $\mathbb{P}$, $\mathbb{P}^h \in L^\infty( \Omega, \mathcal{L}(\widehat{\M}^2, 
\widehat{\M}^2)  )$ and $\mathbb{Q} \in L^\infty( \Omega, \mathcal{L}(\widehat{\M}^3, \widehat{\M}^3   )  )$ be given by \eqref{eq:M3-1}, \eqref{eq:M3-2} and \eqref{eq:M4-4}--\eqref{eq:M5-1} respectively, with Lamé moduli $\lambda$, $\mu$ satisfying \eqref{eq:M3-7}.

For every $w \in H^3(\Omega)$, we have
\begin{equation}
\label{eq:M6-1}
	(\mathbb{P} + \mathbb{P}^h  ) D^2 w \cdot  D^2 w
		\geq
	t (t^2+l^2)
	\xi_{\mathbb{P}} |D^2 w |^2
	 \quad \hbox{a.e. in } \Omega,
\end{equation}
\begin{equation}
\label{eq:M6-2}
	\mathbb{Q} D^3 w \cdot D^3 w 
	\geq t^3\mathit{l}^2
	 \xi_{\mathbb{Q}} 
	|D^3 w |^2
 \quad \hbox{a.e. in } \Omega,
\end{equation}
where $\xi_ {\mathbb{P}}, \xi_{\mathbb{Q}}$ are positive constants only depending on $\alpha_0$ and $\gamma_0$, and $l$ has been defined in \eqref{eq:M4-1}.
\end{lem}
In order to introduce the variational formulation of the Neumann problem \eqref{eq:M2-1}--\eqref{eq:M2-4}, we need to derive an expression of the derivatives with respect to the Cartesian variables at boundary points in terms of the derivatives with respect to local variables.
We shall need the following lemma,  whose proof is postponed in the Appendix.

\begin{lem}
	\label{Lemma:Fichera}
	Let $\Omega$ be a bounded domain in $\R^2$ of $C^2$ class and let $w\in H^3(\Omega)$. 
The following change of variables formulas hold on $\partial\Omega$: 
\begin{eqnarray}\label{eq:change_var_first}
w,_\beta=w,_n n_\beta + w,_s \tau_\beta\  \ \ \mbox{a.e. on } \partial\Omega,
\end{eqnarray}
\begin{eqnarray}
	\label{eq:change_var_secondBIS}
&&w,_{\alpha\beta}=w,_{ss}\tau_\alpha\tau_\beta+w,_{nn}n_\alpha n_\beta+
w,_{sn}(\tau_\alpha n_\beta+\tau_\beta n_\alpha) +\nonumber\\
&&+w,_{s}(\tau_\beta \tau_\alpha,_s-n_\beta n_\alpha,_s)+w,_{n}\tau_\beta n_\alpha,_s \  \ \ \mbox{a.e. on } \partial\Omega.
\end{eqnarray}
\end{lem}

We are now in position to deduce the weak formulation of the problem  \eqref{eq:M2-1}--\eqref{eq:M2-4}.
By multiplying equation \eqref{eq:M2-1} by a test function $w \in H^3(\Omega)$ and integrating by parts three times we get 
\begin{eqnarray}\label{eq:derivingvar_1}
&&\int_{\Omega} -M_{\alpha\beta}(u) w_{,\alpha\beta} + \overline{M}_{\alpha \beta \gamma}^h(u)w,_{\alpha \beta \gamma}= 
\int_{\partial \Omega} ( M_{\alpha\beta}(u)  + \overline{M}_{\alpha \beta \gamma,\gamma}^h(u))_{,\alpha} n_{\beta} w - \ \ \ \ \ \ \ \ \ \nonumber \\ 
&& -\int_{\partial \Omega} ( M_{\alpha\beta}(u)  + \overline{M}_{\alpha \beta \gamma,\gamma}^h(u)) n_{\alpha} w_{,\beta} + \int_{\partial \Omega}  \overline{M}_{\alpha \beta \gamma}^h(u)n_{\gamma}w_{,\alpha\beta}.
\end{eqnarray}
By using formulas \eqref{eq:change_var_first} and \eqref{eq:change_var_secondBIS} in the second and third boundary integral on the right hand side of \eqref{eq:derivingvar_1} respectively, we end up with
\begin{eqnarray*}
&&\int_{\Omega} -M_{\alpha\beta}(u) w_{,\alpha\beta} + \overline{M}_{\alpha \beta \gamma}^h(u)w,_{\alpha \beta \gamma}= \\
&&=\int_{\partial\Omega} ( M_{\alpha\beta}(u)  + \overline{M}_{\alpha \beta \gamma,\gamma}^h(u))_{,\alpha} n_{\beta}w +\\
&&+\int_{\partial\Omega}[ ( M_{\alpha\beta}(u)  + \overline{M}_{\alpha \beta \gamma,\gamma}^h(u))n_{\alpha}\tau_{\beta} + (\overline{M}_{\alpha \beta \gamma}^h(u)n_{\gamma}\tau_{\alpha}\tau_{\beta})_{,s} - \overline{M}_{\alpha \beta \gamma}^h(u)n_{\gamma}(\tau_{{\alpha}_{,s}} \tau_{\beta} -n_{{\alpha}_{,s}} n_{\beta}) ]_{,s} w +\\
&&+\int_{\partial \Omega} \{- ( M_{\alpha\beta}(u)  + \overline{M}_{\alpha \beta \gamma,\gamma}^h(u)) n_{\alpha} n_{\beta}   + \overline{M}_{\alpha \beta \gamma}^h(u))n_{\gamma}n_{\alpha,s}\tau_{\beta} -[\overline{M}_{\alpha \beta \gamma}^h(u)) n_{\gamma}(\tau_{	\alpha}n_{\beta} + \tau_{\beta}n_{\alpha})]_{,s}\}w_{,n}+\\
&&+\int_{\partial\Omega} \overline{M}_{\alpha \beta \gamma}^h(u))n_{\alpha}n_{\beta} n_{\gamma} w_{,nn}\ .
\end{eqnarray*}
Hence taking into account the boundary conditions  \eqref{eq:M2-2}--\eqref{eq:M2-4}, the latter implies that 
\begin{equation}\label{eq:derivingvar_2}
\int_{\Omega} -M_{\alpha\beta}(u) w_{,\alpha\beta} + \overline{M}_{\alpha \beta \gamma}^h(u)w,_{\alpha \beta \gamma}= 
-\int_{\partial \Omega} \widehat{V} w + \widehat{M}_n w,_{n} + \widehat{M}^h_n w,_{nn}. 
\end{equation}
By expressing the derivative  $w_{,n}$ and $w_{,nn}$ at the boundary with respect to Cartesian coordinates, we can rewrite \eqref{eq:derivingvar_2} as follows 
\begin{equation}\label{eq:derivingvar_3}
\int_{\Omega} -M_{\alpha\beta}(u) w_{,\alpha\beta} + \overline{M}_{\alpha \beta \gamma}^h(u)w,_{\alpha \beta \gamma}= -
\int_{\partial \Omega} \widehat{V} w + \widehat{M}_n n_{\alpha}w,_{\alpha} + \widehat{M}^h_n n_{\alpha}n_{\beta} w,_{\alpha\beta}.
\end{equation}
Choosing as test function $w=1, \ w=x_1, \ w=x_2$ in \eqref{eq:derivingvar_3}, we obtain the following three compatibility conditions 
\begin{eqnarray}\label{eq:1:comp}
\int_{\partial \Omega} \widehat{V} =0\ , \  \int_{\partial \Omega}  \widehat{V} x_1 +  \widehat{M}_n n_1=0\ , \ \int_{\partial \Omega}  \widehat{V} x_2 +  \widehat{M}_n n_2=0.
\end{eqnarray}

Let us denote 
\begin{eqnarray}\label{bilinear}
a(u,w)=\int_{\Omega} -M_{\alpha\beta}(u)w,_{\alpha \beta} + \overline{M}_{\alpha \beta \gamma}^h(u)w,_{\alpha \beta \gamma}, 
\end{eqnarray}
\begin{eqnarray}\label{functionalL}
L(w)=-\int_{\partial \Omega}  \widehat{V} w + \widehat{M}_n w,_{n} + \widehat{M}^h_n w,_{nn},
\end{eqnarray}
\begin{eqnarray}\label{functionalLtilde}
\tilde{L}(w)=-
\int_{\partial \Omega} \widehat{V} w + \widehat{M}_n n_{\alpha}w,_{\alpha} + \widehat{M}^h_n n_{\alpha}n_{\beta} w,_{\alpha\beta}.
\end{eqnarray}

\noindent The variational formulation of the Neumann problem \eqref{eq:M2-1}--\eqref{eq:M2-4} is as follows.

\begin{definition}(Weak formulation of the Neumann problem)

A function $w\in H^3(\Omega)$ satisfying 
\begin{eqnarray}
a(u,w)=L(w), \ \ \ \mbox{for every} \ w\in H^3(\Omega),
\end{eqnarray}
is called a weak solution to the Neumann problem \eqref{eq:M2-1}--\eqref{eq:M2-4}.

\end{definition}
From this definition it is evident that, given a weak solution $u$, also $u+l$ is a solution, for every affine function $l$.
Therefore, in order to uniquely identify the solution, we assume the following normalization conditions
\begin{eqnarray}\label{eq: normalization}
\int_{\Omega} u=0, \quad \int_{\Omega} u_{,\alpha}=0, \quad \alpha=1,2.
\end{eqnarray}
\begin{prop} [\bf Well-posedness of the Neumann problem]\label{prop:DirectProblem}
Let $\Omega$ be a bounded domain in $\mathbb{R}^2$ with boundary $\partial \Omega$ of class $C^{2,1}$ with constant $r_0,M_0$.
Let the tensors $\mathbb{P}$, $\mathbb{P}^h \in L^\infty( \Omega, \mathcal{L}(\widehat{\M}^2, \widehat{\M}^2   )  )$ and $\mathbb{Q} \in L^\infty( \Omega, \mathcal{L}(\widehat{\M}^3, \widehat{\M}^3   )  )$ be given by \eqref{eq:M3-1}, \eqref{eq:M3-2} and \eqref{eq:M4-4}--\eqref{eq:M5-1} respectively, with Lamé moduli $\lambda$, $\mu$ satisfying \eqref{eq:M3-7}.
Let $\widehat{V} \in H^{ - 5/2  }(\partial \Omega), \quad 
	\widehat{M}_n \in H^{ - 3/2  }(\partial \Omega), \quad \widehat{M}_n^h \in H^{ - 1/2  }(\partial \Omega)$ such that the compatibility conditions \eqref{eq:1:comp} are satisfied. Problem \eqref{eq:M2-1} - \eqref{eq:M2-4} admits a unique weak solution $u \in H^3(\Omega)$ satisfying \eqref{eq: normalization}. Moreover, 
	\begin{equation}
		\label{eq:reg-estim-H3}
	\|u\|_{H^3(\Omega)}\le C \left(\|\widehat{V}\|_{H^{ - 5/2  }(\partial \Omega)} + r_0^{-1}\| \widehat{M}_n\|_{H^{ - 3/2  }(\partial \Omega)} + r_0^{-2}\|  \widehat{M}_n^h\|_{H^{ - 1/2  }(\partial \Omega)} \right)
	\end{equation}
	where $C>0$ only depends on $M_0$, $M_1$, $\frac{t}{r_0}$, $\frac{l}{r_0}$, $\xi_ {\mathbb{P}}, \xi_{\mathbb{Q}}$ (as defined in Lemma \ref{Lemma:M6-1}). Furthermore, any weak solution to problem \eqref{eq:M2-1}--\eqref{eq:M2-4} is of the form $u+e$, where $e$ is an affine function. 
\end{prop}

\begin{proof}
We introduce the subspace $H(\Omega)$ of $H^3(\Omega)$ defined by 
\begin{equation}
H(\Omega)=\{v \in H^3(\Omega) : \int_{\Omega}v=0,\quad \int_{\Omega} v,_{\alpha}=0,\quad \alpha=1,2   \}, 
\end{equation} 
endowed with the usual $\|\cdot\|_{H^3(\Omega)}$ norm.
By the standard  Poincar\'{e} inequality (see for instance \cite[Proposition 3.3]{MRV2007}) we have that 
\begin{equation}\label{eq:Poincare}
{r_0}^2\int_{\Omega} |{D}^2 v |^2 + {r_0}^4\int_{\Omega} |{D}^3 v |^2 \le \| v\|_{H^3(\Omega)}^2 \le C\left(  {r_0}^2\int_{\Omega} |{D}^2 v |^2 + {r_0}^4\int_{\Omega} |{D}^3 v |^2 \right)
\end{equation}
where $C>0$ is a constant only depending on $M_0,M_1$. 

We consider the continuous bilinear form 
\begin{eqnarray}
a: H(\Omega)\times H(\Omega) \rightarrow \mathbb{R},
\end{eqnarray}
where $a$ is defined in \eqref{bilinear}. 
By Lemma \ref{Lemma:M6-1} and \eqref{eq:Poincare} we have that, for every $w\in H(\Omega)$, 
\begin{equation}\label{coercive}
a(w,w)\ge C \left(r_0^3\int_{\Omega} |D^2 w|^2 + r_0^5\int_{\Omega}|D^3 w|^2\right)\ge Cr_0 \|w\|_{H^3(\Omega)}^2,
\end{equation}
with $C>0$ only depending on $\frac{t}{r_0}$, $\frac{l}{r_0}$, $\xi_ {\mathbb{P}}, \xi_{\mathbb{Q}}, M_0,M_1$.
Hence, we may infer that the bilinear form $a$ is coercive. 
By standard trace inequalities we have that for any $w\in H(\Omega)$ 
\begin{equation}\label{continuousfunctional}
|\tilde{L}(w)|\le Cr_0\left(\|\widehat{V}\|_{H^{-5/2}(\partial\Omega)} + r_0^{-1}\|\widehat{M}_n\|_{H^{-3/2}(\partial\Omega)} +  r_0^{-2}\|\widehat{M}^h_n\|_{H^{-1/2}(\partial\Omega)} \right)\|w\|_{H^3(\Omega)},
\end{equation}
with $C>0$ only depending on $M_0$, $M_1$. {}From the latter we deduce that $\tilde{L}$ is a continuous functional on $H(\Omega)$. 
By the Lax-Milgram Theorem, we can infer that there exists a unique $u\in H(\Omega)$ such that 
\begin{equation}\label{LM}
a(u,w)=\tilde{L}(w), \quad \mbox{for any} \ w\in H(\Omega).
\end{equation}
Given any $g\in H^3(\Omega)$, there exists $w\in H(\Omega)$ and an affine function $e(x_1,x_2)=a +bx_1 +cx_2$ such that 
\begin{eqnarray}
g(x_1,x_2)=w(x_1,x_2) +a +bx_1 +cx_2 \ .
\end{eqnarray}
By using the compatibility conditions \eqref {eq:1:comp}, \eqref{LM} extends to every test function $g\in H^3(\Omega)$, that is $u$ is the desired weak solution to problem \eqref{eq:M2-1}--\eqref{eq:M2-4}.
By using the weak solution $u$ as test function in \eqref{LM}  and  combining \eqref{coercive}, \eqref{continuousfunctional} we get 
 \eqref{eq:reg-estim-H3}.
Finally, let us assume that $v\in H^3(\Omega)$ is a weak solution to \eqref{eq:M2-1}--\eqref{eq:M2-4}. Hence $a(u-v,w)=0,\ \  \mbox{for any}\ \ w\in H^3(\Omega)$
and choosing as test function $w=u-v$ and by Lemma \ref{Lemma:M6-1} we deduce that $\|D^2(u-v)\|_{L^2(\Omega)}=0$, meaning that $u-v$ is an affine function a.e. in $\Omega$. 
\end{proof}
\subsection{Advanced regularity}
\label{subsec:regularity}

We conclude Section $3$ with a global regularity result. 
\begin{theo}[\bf Global $H^4$-regularity]
	\label{prop:glo-reg-res-exProp8-2}
	Let $\Omega$ be a bounded domain in $\R^2$ with boundary
	$\partial \Omega$ of class $C^{3,1}$ with constants $r_0$, $M_0$, and satisfying \eqref{eq:M1-1}. Let $u \in H^3(\Omega)$ be the weak solution of the Neumann problem
	\eqref{eq:M2-1}--\eqref{eq:M2-4} satisfying \eqref{eq: normalization}, where $\widehat{V} \in H^{ - 3/2  }(\partial \Omega), \quad 
	\widehat{M}_n \in H^{ - 1/2  }(\partial \Omega), \quad \widehat{M}_n^h \in H^{  1/2  }(\partial \Omega)$ are such that the compatibility conditions \eqref{eq:1:comp} are satisfied. Assume that $\mathbb P$, $\mathbb P^h$ defined in \eqref{eq:M3-1}, \eqref{eq:M3-2} are 
	of class $C^{0,1}( \overline{\Omega})$ and satisfy the
	strong convexity condition \eqref{eq:M6-1}. Moreover, let us assume that $\mathbb{Q}$, defined in \eqref{eq:M4-4}, is 
	of class $C^{0,1}( \overline{\Omega})$ and satisfies the
	strong convexity condition \eqref{eq:M6-2}.
	
	Then $u \in H^4(\Omega)$ and
	\begin{equation}
		\label{eq:GR1-1}
		\| u \|_{H^4(\Omega)} \leq C \left(\|\widehat{V}\|_{H^{ - 3/2  }(\partial \Omega)} + r_0^{-1}\| \widehat{M}_n\|_{H^{ - 1/2  }(\partial \Omega)} + r_0^{-2}\|  \widehat{M}_n^h\|_{H^{  1/2  }(\partial \Omega)} \right),
	\end{equation}
	where $C>0$ only depends on $M_0$, $M_1$, $\frac{t}{r_0}$, $\frac{l}{r_0}$, $\xi_{\mathbb Q}$, $\xi_{\mathbb P}$, $\|\mathbb P\|_{C^{0,1}( \overline{\Omega})}$, $\|\mathbb P^h\|_{C^{0,1}( \overline{\Omega})}$, $\|\mathbb Q\|_{C^{0,1}( \overline{\Omega})}$.
\end{theo}
The proof of Theorem \ref{prop:glo-reg-res-exProp8-2} is based on the following two results, the proof of which is given at the end of this section.

\begin{theo} [{\bf Interior regularity}]
	\label{theo:loc-int-reg-res-exTheo8-3}
	Let $B_{\sigma}$ be an open ball in $\R^2$ centered
	at the origin and with radius $\sigma$. Let $u \in H^3(B_{\sigma})$ be such that
	\begin{equation}
		\label{eq:IR1-1}
		a(u, \varphi) = 0,		   \qquad \hbox{for every }
		\varphi \in H^3_0(B_{\sigma}),
	\end{equation}
	where
	\begin{equation}
		\label{eq:IR1-2}
		a(u, \varphi) = \int_{B_{\sigma}}  (\mathbb P + \mathbb P^h) D^2u \cdot D^2 \varphi + \mathbb Q D^3 u \cdot D^3 \varphi.
	\end{equation}
 The tensors $\mathbb P, \mathbb P^h \in C^{0,1} (\overline{B_{\sigma}})$, $\mathbb Q \in C^{0,1} (\overline{B_{\sigma}})$ defined in \eqref{eq:M3-1}, \eqref{eq:M3-2}, \eqref{eq:M4-4} satisfy the strong convexity conditions \eqref{eq:M6-1}, \eqref{eq:M6-2}, respectively. 
	
	Then, $u \in H^4( B_{  \frac{\sigma}{2}  })$ and we have
	\begin{equation}
		\label{eq:IR1-4}
		\|u\|_{ H^4 (B_{\frac{\sigma}{2}})} \leq C  \|u\|_{ H^3 (B_{\sigma})},
	\end{equation}
	where $C>0$ only depends on $\frac{t}{r_0}$, $\frac{l}{r_0}$, $\xi_{\mathbb Q}$, $\xi_{\mathbb P}$, $\|\mathbb P\|_{C^{0,1}( \overline{B_\sigma})}$, $\|\mathbb P^h\|_{C^{0,1}( \overline{B_\sigma})}$, $\|\mathbb Q\|_{C^{0,1}( \overline{B_\sigma})}$.
\end{theo}
\begin{theo} [{\bf Boundary regularity}]
	\label{theo:loc-bndry-reg-res-exTheo8-4}
	Let us denote by $B_{\sigma}^+$ the hemidisk
	$\{(x_1,x_2) \in \R^2 | \ x_1^2+x_2^2 < \sigma^2, \ x_2>0 \}$ and
	let $\Gamma_{\sigma}= \{(x_1,x_2) \in \R^2 | \ -\sigma<x_1<\sigma, \ x_2=0
	\}$, $\Gamma_{\sigma}^+ = \partial B_{\sigma}^+ \setminus
	\Gamma_{\sigma}$. Let $u \in H^3(B_{\sigma}^+)$ be such that
	\begin{equation}
		\label{eq:BR1-1}
		a_+(u, \varphi) = \textit{l}_+(\varphi),  \qquad \hbox{for every }
		\varphi \in H^3_{ \Gamma_{\sigma}^+}  (B_{\sigma}^+),
	\end{equation}
	where $H^3_{ \Gamma_{\sigma}^+}  (B_{\sigma}^+)= \{ g \in H^3
	(B_{\sigma}^+) | \  g=0, \ \frac{\partial g}{\partial n}=0, \ \frac{\partial^2 g}{\partial n^2}=0 \ on \
	\Gamma_{\sigma}^+ \}$,
	\begin{equation}
		\label{eq:BR1-2}
		a_+(u, \varphi) = a_+^{\mathbb E}(u, \varphi) +  a_+^{\mathbb K}(u, \varphi),
	\end{equation}
	\begin{equation}
		\label{eq:BR1-3}
		a_+^{\mathbb E}(u, \varphi) =
		\int_{B_1^+} \sum_{i,j=1}^2 \mathbb E^{(i,j)} D^i u \cdot D^j \varphi,
	\end{equation}
	\begin{equation}
		\label{eq:BR1-4}
		a_+^{\mathbb K}(u, \varphi) =
		\int_{B_1^+} \sum_{i,j=1}^3 \mathbb K^{(i,j)} D^i u \cdot D^j \varphi,
	\end{equation}
	and $\textit{l}_+(\cdot)$ is a continuous functional on $H^3_{ \Gamma_{\sigma}^+}  (B_{\sigma}^+)$ such that
	\begin{equation}
		\label{eq:l+}
		| \textit{l}_+(\varphi)| \leq G\| \varphi \|_{ H^2 (B_\sigma^+)   }
		\qquad \hbox{for every }
		\varphi \in H^3_{ \Gamma_{\sigma}^+}  (B_{\sigma}^+),
	\end{equation}
	where $G$ is a positive constant and $ \mathbb E^{(i,j)} \in \mathcal{L}( \mathbb{M}^i, \mathbb{M}^j )$, $i,j=1,2$ ($\mathbb{M}^1 \equiv \R^2$), $ \mathbb K^{(i,j)} \in \mathcal{L}( \mathbb{M}^i, \mathbb{M}^j )$, $i,j=1,2,3$, $f \in L^2(B_{\sigma}^+)$. Let the tensor fields  $ \mathbb E^{(i,j)} $, $\mathbb K^{(i,j)} $ be of $C^{0,1}$ class in $\overline{B_{\sigma}^+}$ satisfying
	\begin{equation}
		\label{eq:BR1-6}
		\sum_{i,j=1}^2 \sigma^{4-(i+j)} \| \mathbb E^{(i,j)}\|_{C^{0,1} (\overline{B_{\sigma}^+)  } }\leq E,
	\end{equation}
	\begin{equation}
		\label{eq:BR2-1}
		\sum_{i,j=1}^3 \sigma^{6-(i+j)} \| \mathbb K^{(i,j)}\|_{C^{0,1} (\overline{B_{\sigma}^+)  } }\leq K,
	\end{equation}
	for some positive constants $E$, $K$. Moreover, let $\mathbb E^{(2,2)}$ and $\mathbb K^{(3,3)}$ satisfy the symmetry conditions \eqref{eq:M3-3} and \eqref{eq:M5-3}, respectively, and the strong convexity conditions  $\mathbb E^{(2,2)} D^2 w \cdot  D^2 w \geq \xi_{\mathbb{E}} 
	|D^2 w |^2$, 
	$\mathbb K^{(3,3)}D^3 w \cdot D^3 w \geq \xi_{\mathbb{K}} 
	|D^3 w |^2$  in $B_{\sigma}^+$, for every $w \in H^3(B_{\sigma}^+)$, where $\xi_{\mathbb{E}}$, $\xi_{ \mathbb{K}}$ are positive constants.
		
	Then $u \in H^4( B_{ \frac{\sigma}{2}  }^+)$ and we have
	\begin{equation}
		\label{eq:BR2-2}
		\|u\|_{ H^4 (B_{\frac{\sigma}{2}}^+)} \leq C \left (
		G + \|u\|_{ H^3 (B_{\sigma}^+)} \right ),
	\end{equation}
	where $C>0$ only depends on $\xi_\mathbb{K}$, $\xi_{\mathbb{E}}$, $E$, $K$.
\end{theo}

\begin{proof}[Proof of Theorem \ref{prop:glo-reg-res-exProp8-2}]

Without loss of generality we can assume $r_0=1$.

By the regularity of $\partial \Omega$, we can construct a finite
collection of open sets $\Omega_0$, $\widetilde{\Omega_1}$, ..., $\widetilde{\Omega_N}$ and, for every $j$, $j=1,...,N$, a homeomorphism ${\cal
	{T}}_{(j)}$ of $C^{3,1}$ class which maps $\Omega_j= \widetilde{\Omega_j}  \cap \Omega $ into
$B_{1}^+$, $ \overline{\Omega_j} \cap \partial \Omega$ into
$\Gamma_1$ and $\partial \Omega_j \cap \Omega$ into $\Gamma_1^+$,
such that $\Omega = \Omega_0 \cup \left ( \cup_{j=1}^N {\cal
	{T}}_{(j)}^{-1}(B_{ \frac{1}{2}}^+) \right )$, $\Omega_0 \subset
\Omega_{\delta_0}$, where $\delta_0>0$ only depends on $M_0$. Note that here we have used the notation introduced in Theorem \ref{theo:loc-bndry-reg-res-exTheo8-4} for $\Gamma_1$ and $\Gamma_1^+$. By the regularity of $\partial
\Omega$ and \eqref{eq:M1-1}, the number $N$ is controlled by a
constant only depending on $M_0$ and $M_1$.

The set $\Omega_0$ can be covered by a finite number of balls contained
in $\Omega$. Therefore, using the local interior regularity result of
Theorem \ref{theo:loc-int-reg-res-exTheo8-3}, we have that $u \in H^4
(\Omega_0)$ and
\begin{equation}
	\label{eq:GR2-1}
	\|u\|_{ H^4 (\Omega_0)} \leq C \|u\|_{ H^3 (\Omega)},
\end{equation}
where $C>0$ only depends on $t$, $l$, $M_1$,  $\xi_{\mathbb Q}$, $\xi_{\mathbb P}$, $\|\mathbb P\|_{C^{0,1}( \overline{\Omega})}$, $\|\mathbb P^h\|_{C^{0,1}( \overline{\Omega})}$, $\|\mathbb Q\|_{C^{0,1}( \overline{\Omega})}$.

We now fix $j$, $1 \leq j \leq N$, and we determine an
estimate analogous to \eqref{eq:GR2-1} in $\Omega_j$.

Let us define
\begin{equation}
	\label{eq:GR2-2}
	H^3_{ \partial \Omega_j \cap \Omega   } (\Omega_j)
	=
	\left \{
	f \in H^3(\Omega_j) | \ h=h_{,n}=h_{,nn}=0 \ \hbox{on} \ \partial \Omega_j \cap \Omega
	\right \}.
\end{equation}
The function $u\in H^3(\Omega)$, solution of \eqref{eq:M2-1}--\eqref{eq:M2-4}, satisfies 
\begin{equation}
	\label{eq:GR3-1}
	\int_{\Omega_j}  (\mathbb P + \mathbb P^h) D^2u \cdot D^2 \varphi + \mathbb Q D^3 u \cdot D^3 \varphi = 
	L_+(\varphi) \quad \hbox{for every } \varphi \in H^3_{ \partial \Omega_j \cap \Omega   } (\Omega_j),
\end{equation}
where
\begin{equation}
	\label{eq:funz-dati-bordo+}
	L_+(\varphi)
	=
	-\int_{ \overline{\Omega_j} \cup \partial \Omega   }  \widehat{V} \varphi + \widehat{M}_n \varphi,_{n} + \widehat{M}^h_n \varphi,_{nn}.
\end{equation}
Hereinafter, to simplify the notation, we denote ${\cal
	{T}}_{(j)}$ by $\cal{T}$, and we define
\begin{equation}
	\label{eq:GR3-2}
	y= {\cal T} (x), \qquad y \in B_{1}^+,
\end{equation}
\begin{equation}
	\label{eq:GR3-3}
	x= {\cal T}^{-1} (y), \qquad x \in \Omega_j,
\end{equation}
\begin{equation}
	\label{eq:GR3-4}
	v(y) = u ( {\cal T}^{-1} (y)),
\end{equation}
\begin{equation}
	\label{eq:GR3-5}
	\psi(y) = \varphi({\cal T}^{-1} (y)), \qquad \psi \in H_{\Gamma_1^+}^3(B_1^+).
\end{equation}
By changing the variables in \eqref{eq:GR3-1} according to \eqref{eq:GR3-2}, the function $v$ belongs to $H^3(B_1^+)$ and satisfies
\begin{multline}
	\label{eq:GR3-7}
	\int_{B_1^+} \sum_{i,j=1}^2 \mathbb E^{(i,j)} D^i v \cdot D^j \psi +
	\sum_{i,j=1}^3 \mathbb K^{(i,j)} D^i v \cdot D^j \psi
	=
	 \mathcal{L}_+ ( \psi) \quad \quad {\mbox{for every}} \  \psi\in H^3_{ \Gamma_1^+} (B_1^+),
\end{multline}
where
\begin{multline}
	\label{eq:dato-bordo-trasf}
	\mathcal{L}_+ (\psi) =
	\\
	=
	-
	\int_{\Gamma_1}
	\left ( 
	\widehat{\mathcal{V}}\psi + \widehat{\mathcal{M}}_n S^T D \psi \cdot S^T \nu | S^{-T} n| +
	\widehat{\mathcal{M}}_n^h
	\left (  R D \psi +S^T D^2 \psi S    \right )S^T \nu \cdot S^T \nu | S^{-T} n|^2
	\right )
	|S \tau|^{-1},
\end{multline}
with
\begin{equation}
	\label{eq:dato-bdry-V}
	\widehat{\mathcal{V}}(y) = \widehat{V}({\cal T}^{-1} (y)),
\end{equation}
\begin{equation}
	\label{eq:dato-bdry-Mn}
	\widehat{\mathcal{M}}_n= \widehat{{M}}_n({\cal T}^{-1} (y)),
\end{equation}
\begin{equation}
	\label{eq:dato-bdry-Mnh}
	\widehat{\mathcal{M}}_n^h = \widehat{{M}}_n^h ( {\cal T}^{-1} (y))
\end{equation}
and
\begin{equation}
	\label{eq:GR4-3_e_4-5}
	S_{kr} = \frac{\partial {\cal T}_k}{\partial x_r}, \qquad
	R_{ksr} = \frac{\partial^2  {\cal T}_{k}  }{\partial x_s \partial
		x_r}.
\end{equation}
Here, $\nu$ is the outer unit normal to $B_1^+$. As shown in \cite{MRV2007} (Proposition $8.2$), the expressions of the tensors $\mathbb E^{(i,j)}\in \mathcal{L}( \mathbb{M}^i, \mathbb{M}^j )$, $i,j=1,2$, ($\mathbb{M}^1 \equiv \R^2$) can be deduced passing to Cartesian coordinates, namely
\begin{multline}
	\label{eq:GR5-2}
	\sum_{i,j=1}^2 \mathbb E^{(i,j)} D^i v \cdot D^j \psi
	=
	\\
	=
	(P_{ijrs}+P_{ijrs}^h) 
	(S_{kr}S_{ls}v_{,kl}+R_{krs}v_{,k})
	(S_{mi}S_{nj}\psi_{,mn}+R_{nij}\psi_{,n})  \left | \det S   \right |^{-1} .
\end{multline}
By the regularity assumptions on $\mathbb P$, $\mathbb P^h$ and the regularity of the boundary $\partial \Omega$, the tensors $\mathbb{E}^{(i,j)}$ belong to $C^{0,1}(\overline{B_1^+})$, $i,j=1,2$. Moreover, by the properties of $\mathbb{P}$ and $\mathbb{P}^h$, the fourth order tensor $\mathbb{E}^{(2,2)}$ satisfies the symmetry conditions 
\begin{equation}
	\label{eq:GR-sym-cnds-E22}
	E_{mnkl}^{(2,2)} = E_{klmn}^{(2,2)}=E_{klnm}^{(2,2)},  \quad m,n,k,l=1,2, \hbox{ in } B_1^+
\end{equation}
and the strong convexity condition
\begin{equation}
	\label{eq:GR-ell-convex-E22}
	\mathbb{E}^{(2,2)} A \cdot A \geq \xi_{\mathbb{E}}^* |A|^2 \hbox{ in }
	\overline{B_1^+},
\end{equation}
for every $2 \times 2$ symmetric matrix $A$, where $ \xi_{\mathbb{E}}^* >0$ is a constant only depending on $t$, $l$, $M_0$ and $\xi_{\mathbb{E}}$. 

The term in \eqref{eq:GR3-7} involving the tensors $ \mathbb K^{(i,j)} \in \mathcal{L}( \mathbb{M}^i, \mathbb{M}^j )$ can be analysed similarly. We have
\begin{eqnarray}
	\label{eq:GR10-1}
	&&\sum_{i,j=1}^3 \mathbb K^{(i,j)} D^i v \cdot D^j \psi
	=
	\\
	&=&Q_{ijklmn} 
	(
	S_{\alpha l} S_{\beta m} S_{\gamma n} v_{,\alpha \beta \gamma}
	+
	T_{lmn \alpha \beta} v_{,\alpha \beta} 
	+
	Z_{lmn\alpha} v_{,\alpha}
	) 	\cdot 
	(
	S_{\delta i} S_{\epsilon j} S_{\vartheta k} v_{, \delta \epsilon \vartheta}
	+
	T_{ijk \tau \chi}\psi_{,\tau \chi} + Z_{ijk \tau} \psi_{,\tau}
	)
	\left | \det S \right |^{-1},\nonumber
\end{eqnarray}
where 
\begin{equation}
	\label{eq:GR_def_T_5index}
	T_{lmn \alpha \beta}= R_{\alpha ln}S_{\beta m}+R_{\beta m n}S_{\alpha l}+
	R_{\alpha m l}S_{\beta n},
\end{equation}
\begin{equation}
	\label{eq:GR_def_Z_4index}
		Z_{lmn\alpha} = \frac{\partial^3  {\cal T}_{\alpha}  }{\partial x_l \partial
		x_m \partial x_n}.
\end{equation}
By the regularity assumptions on $\mathbb Q$ and on the boundary $\partial \Omega$, the tensors $\mathbb{K}^{(i,j)}$ belong to $C^{0,1}(\overline{B_1^+})$, $i,j=1,3$. 

The sixth order tensor $\mathbb{K}^{(3,3)}$ satisfies the symmetry conditions \eqref{eq:M5-3} and the strong convexity condition 
\begin{equation}
	\label{eq:GR11-1}
	\hbox{ (Claim A)  }
	\quad
	\mathbb{K}^{(3,3)} A \cdot A \geq \xi_{ \mathbb{K}}^* |A|^2 \quad \hbox{in }    \overline{B_1^+},
\end{equation}
for every $A \in \widehat{\mathbb{M}}^{3}$, where $ \xi_{\mathbb{K}}^* >0$ is a constant only depending on $t$, $l$, $M_0$ and $\xi_{\mathbb{Q}}$. A proof of Claim $A$ is presented at the end of this proof.

By the regularity of the Neumann data, by the Poincaré inequality and by trace inequalities, we have
\begin{multline}
	\label{eq:controllo-funz-dati-trasf}
	| \mathcal{L}_+(\psi)| 
	\leq
	C
	\left (
	\|\widehat{\mathcal{V}}\|_{H^{ - 3/2  }(\Gamma_1)} 
	\|\psi\|_{H^{ 3/2  }(\Gamma_1)} 
	+
	\| \widehat{\mathcal{M}}_n\|_{H^{ - 1/2  }(\Gamma_1)} 
	\| D \psi\|_{H^{ 1/2  }(\Gamma_1)} 
	+
	\right.
	\\
	\left.
	+
	\|  \widehat{\mathcal{M}}_n^h\|_{H^{  1/2  }(\Gamma_1)} 
	\left ( \|  D \psi \|_{H^{ - 1/2  }(\Gamma_1)} +
	\|  D^2 \psi \|_{H^{ - 1/2  }(\Gamma_1)}
	\right )
	\right )
	\leq	
	\\
	\leq C
	\left(
	\|\widehat{\mathcal{V}}\|_{H^{ - 3/2  }(\Gamma_1)} + \| \widehat{\mathcal{M}}_n\|_{H^{ - 1/2  }(\Gamma_1)} + \|  \widehat{\mathcal{M}}_n^h\|_{H^{  1/2  }(\Gamma_1)} 
	\right)
	\| \psi \|_{ H^2 (B_1^+)   },
\end{multline}
for every  $\psi \in H^3_{\Gamma_1^+}(B_1^+)$,  where $C>0$ only depends on $M_0$.

Finally, by the regularity result up to the boundary, see Theorem \ref{theo:loc-bndry-reg-res-exTheo8-4}, we have that $v \in H^4( B_{ \frac{1}{2}   }^+   )$ and 
\begin{equation}
	\label{eq:GR-bdry-loc-v}
	\|v\|_{ H^4 (B_{\frac{1}{2}}^+)} \leq C \left (
	\|\widehat{\mathcal{V}}\|_{H^{ - 3/2  }(\Gamma_1)} + \| \widehat{\mathcal{M}}_n\|_{H^{ - 1/2  }(\Gamma_1)} + \|  \widehat{\mathcal{M}}_n^h\|_{H^{  1/2  }(\Gamma_1)} + \|v\|_{ H^3 (B_1^+)} \right ),
\end{equation}
and, by applying the
homeomorphism ${\cal {T}}$, we have
\begin{multline}
	\label{eq:GR-bdry-loc-u}
	\|u\|_{ H^4 (\Omega_{j,{\frac{1}{2}}})} 
	\leq 
	\\
	\leq C \left (
	\|{\widehat{V}}\|_{H^{ - 3/2  }(\overline{\Omega_j} \cup \partial \Omega)} + \| {\widehat{M}}_n\|_{H^{ - 1/2  }(\overline{\Omega_j} \cup \partial \Omega)} + \|  {\widehat{M}}_n^h\|_{H^{  1/2  }(\overline{\Omega_j} \cup \partial \Omega)} + \|u\|_{ H^3 (\Omega_j)} \right ),
\end{multline}
where $\Omega_{j,{\frac{1}{2}}} = {\cal
	{T}}^{-1}(B_{\frac{1}{2}}^+)$ and $C>0$ only
depends on $t$, $l$, $M_0$, $M_1$, $\xi_{\mathbb Q}$, $\|\mathbb P\|_{C^{0,1}(\overline{\Omega})}$, $\|{\mathbb {P}}^h\|_{C^{0,1}(\overline{\Omega})}$, $\|\mathbb Q\|_{C^{0,1}(\overline{\Omega})}$. Then, estimate \eqref{eq:GR1-1} follows by
 \eqref{eq:GR2-1}, \eqref{eq:GR-bdry-loc-u} and \eqref{eq:reg-estim-H3}.
\end{proof}

\begin{proof}[Proof of Claim $A$]
Let us notice that, for every $A \in \widehat{\mathbb{M}}^{3}$, the matrix $\mathcal{A}$ given by
\begin{equation}
	\label{eq:GR11-1bis}
	\mathcal{A}_{lmn}= S_{\alpha l} S_{\beta m} S_{\gamma n} A_{\alpha \beta \gamma}
\end{equation}
belongs to $\widehat{\mathbb{M}}^{3}$. Therefore, by the strong convexity of $\mathbb{Q}$, we have 
\begin{equation}
	\label{eq:GR12-1}
	\mathbb{K}^{(3,3)} A \cdot A 
	=
	\mathbb{Q} \mathcal{A} \cdot \mathcal{A} \left | \det S \right |^{-1}
	\geq
	\xi_{ \mathbb{Q}}^* |\mathcal{A}|^2 \quad \hbox{in }    \overline{B_1^+},
\end{equation}
for every ${A} \in \widehat{\mathbb{M}}^{3}$, where $ \xi_{\mathbb{Q}}^* >0$ is a constant only depending on $t$, $l$, $M_0$ and $\xi_{\mathbb{Q}}$. To conclude it is enough to prove that there exists a constant $C>0$ such that
\begin{equation}
	\label{eq:GR12-2}
	|\mathcal{A}|^2 \geq C |A|^2 \quad \hbox{in }    \overline{B_1^+},
\end{equation}
for every $A \in \widehat{\mathbb{M}}^{3}$. Noting that $A_{ijk}=\delta_{i\alpha}\delta_{j\beta}\delta_{k\gamma}A_{\alpha \beta\gamma}$ and $\delta_{i\alpha} = ((S^T)^{-1})_{il}S_{\alpha l}$, we have
\begin{equation}
	\label{eq:GR13-1}
	|{A}|^2 = A_{ijk}A_{ijk}= \Theta_{pl}\Theta_{qm}\Theta_{rn}
	\mathcal{A}_{lmn} \mathcal{A}_{pqr},
\end{equation}
where $\Theta=S^{-1}S^{-T}$. By applying Cauchy-Schwarz's inequality iteratively and observing that $| \Theta| \leq C$ in $\overline{B_1^+}$, it is found that $ |A|^2 \leq c |\mathcal{A}|^2$, where $c>0$ only depends on $M_0$, which implies \eqref{eq:GR12-2}.
\end{proof}

\begin{proof}[Proof of Theorem \ref{theo:loc-int-reg-res-exTheo8-3}]
Without loss of regularity we can assume $\sigma=1$. Let $\vartheta \in C_0^{\infty}(\R^2)$ be a function such that $0
\leq \vartheta(x) \leq 1$ in $\R^2$, with $\vartheta
\equiv 1$ in $B_{\rho}$, $\vartheta \equiv 0 $ in $\R^2 \setminus
B_{\sigma_0}$ and $|D^k \vartheta| \leq C$, $k=1,...,4$,
where $\rho=
\frac{1}{2}$, $\sigma_0= \frac{1}{2}(\rho+1)=\frac{3}{4}$  and $C>0$ is an absolute constant.
	
Let $s \in \R\setminus\{0\}$ and let us introduce the difference operator in the $\alpha$th direction as
\begin{equation}
		\label{eq:IR2-1}
	(\tau_{\alpha,s}v)(x) = \frac{v(x+se_{\alpha})-v(x)}{s}, \qquad \alpha=1,2,
\end{equation}
for any function $v$. In what follows we shall assume that $|s|
\leq \frac{1}{16}$.
	
For every function $\varphi \in H_0^3(B_{1})$, let us still denote by
	$\varphi \in H^3(\R^2)$ its extension to the plane $\R^2$ obtained
	by assuming $\varphi \equiv 0$ in $\R^2 \setminus B_{1}$. Let us notice that, for every smooth function $\psi$ and $\alpha=1,2$, we have
\begin{equation}
	\label{eq:IR2-2}
	D^\beta \tau_{\alpha,s} (\psi) = \tau_{\alpha,s} (D^\beta \psi),
\end{equation}
where $D^\beta = D^{\beta_1}_1 D^{\beta_2}_2= \frac{\partial^{\beta_1}}{\partial x_1^{\beta_1}} \frac{\partial^{\beta_2}}{\partial x_2^{\beta_2}}$, $\beta=(\beta_1,\beta_2)$. 

Let us start by elaborating the term in $a(u,\varphi)$ containing the third order derivatives, with $u$ replaced by $\tau_{\alpha,s}(\vartheta u)$. 

We have
\begin{multline*}
	(\vartheta u)_{,\alpha \beta \gamma} =\vartheta u_{,\alpha\beta\gamma}+
	(
	\vartheta_{,\alpha} u_{,\beta \gamma}+  \vartheta_{,\beta} u_{,\alpha \gamma} + \vartheta_{,\gamma} u_{,\alpha \beta}
	)
	+
	(
	\vartheta_{,\alpha\beta} u_{,\gamma}+  \vartheta_{,\alpha\gamma} u_{,\beta} + \vartheta_{,\beta\gamma} u_{,\alpha}
	)
	+
	\vartheta_{,\alpha \beta\gamma} u
\end{multline*}
or, equivalently, in compact notation
\begin{equation}
	\label{eq:IR3-1}
	D^3(\vartheta u)= \vartheta D^3 u + D \vartheta \otimes D^2 u+ D^2 \vartheta \otimes D u + u D^3 \vartheta.
\end{equation}
Therefore, we have
\begin{eqnarray}
	\label{eq:IR3-2}
	&&\int_{B_1} \mathbb Q D^3 (\tau_{\alpha,s} (\vartheta u)   ) \cdot D^3 \varphi =
	\int_{B_1} \mathbb Q ( \tau_{\alpha,s} (\vartheta D^3 u)   ) \cdot D^3 \varphi 
	+ \int_{B_1} \mathbb Q
	(
	\tau_{\alpha,s}
	(
	D \vartheta \otimes D^2 u+ D^2 \vartheta \otimes D u
	)
	)
	\cdot 
	D^3
	\varphi
	+\nonumber
	\\
	&&\quad \quad \quad \quad \quad \quad \quad \quad \quad \quad \quad \quad \quad+
	\int_{B_1} \mathbb Q
	(
	\tau_{\alpha,s}
	(uD^3 \vartheta)
	)
	\cdot 
	D^3
	\varphi 
	\equiv I_1+I_2+I_3.
\end{eqnarray} 
Let us estimate $I_2$. By the definition of $\tau_{\alpha,s}$, we have
\begin{multline*}
\tau_{\alpha,s}
(
D \vartheta \otimes D^2 u+ D^2 \vartheta \otimes D u
)
=
D \vartheta (x+se_\alpha)\otimes \tau_{\alpha,s}
( D^2 u) + \tau_{\alpha,s}(D \vartheta) \otimes D^2 u+
\\
+
D^2 \vartheta (x+se_\alpha)\otimes \tau_{\alpha,s}
( D u) + \tau_{\alpha,s}(D^2\vartheta) \otimes D u
\end{multline*}
and then
\begin{equation}
	\label{eq:IR4-3}
	|I_2| \leq C  \|u\|_{H^3(B_{1})}\|D^3 \varphi\|_{L^2(B_{1})},
\end{equation}
where $C>0$ only depends on $\| \mathbb
Q\|_{L^\infty( \overline{B_{1}})}$. Similarly, we have
\begin{equation}
	\label{eq:IR4-4}
	|I_3| \leq C  \|u\|_{H^1(B_{1})}\|D^3 \varphi\|_{L^2(B_{1})},
\end{equation}
where $C>0$ only depends on $\| \mathbb
Q\|_{L^\infty( \overline{B_{1}})}$.

Let us rewrite the term $I_1$ as follows
\begin{multline}
	\label{eq:IR5-2}
	I_1 = \int_{B_{1}} \mathbb Q ( \tau_{\alpha,s} ( \vartheta D^3 u)) \cdot D^3\varphi =
	\\
	=
	\int_{B_{1}} \tau_{\alpha,s} ( \mathbb Q  ( \vartheta D^3 u)) \cdot D^3\varphi 
	- \int_{B_{1}} (\tau_{\alpha,s} \mathbb Q) (x)  ( \vartheta D^3 u)(x+se_{\alpha}) \cdot D^3\varphi \equiv I_1'
	+ I_1'',
\end{multline}
where
\begin{equation}
	\label{eq:IR6-1}
	|I_1''| \leq C  \|D^3 u\|_{L^2(B_{1})}\|D^3 \varphi\|_{L^2(B_{1})},
\end{equation}
with a constant $C>0$ only depending on $\| \mathbb
Q\|_{C^{0,1}( \overline{B_{1}})}$. By integrating by parts, and recalling that $\varphi \in H_0^3(B_{1} )$, we have
\begin{multline}
	\label{eq:IR7-3}
	I_1' = \int_{B_{1}} \tau_{\alpha,s} ( \mathbb Q   ( \vartheta D^3 u)) \cdot D^3\varphi =
	- \int_{B_{1}} \mathbb Q   ( \vartheta D^3 u) \cdot (\tau_{\alpha,-s} (D^3\varphi)) = \\
	= - \int_{B_{1}} \mathbb Q   (D^3 u) \cdot ( \vartheta \tau_{\alpha,-s} (D^3\varphi)) =
	- \int_{B_{1}} \mathbb Q   (D^3 u) \cdot D^3( \vartheta \tau_{\alpha,-s} \varphi) + \\
	+ \int_{B_{1}} \mathbb Q (D^3 u) 
	\cdot 
	\left ( 
	D^3 \vartheta (\tau_{\alpha,-s} \varphi) +
	D \vartheta \otimes D^2(\tau_{\alpha,-s} \varphi) + 
	 D^2 \vartheta \otimes D(\tau_{\alpha,-s} \varphi)
	\right ) =
	\\
	= - \int_{B_{1}} \mathbb Q   (D^3 u) \cdot D^3( \vartheta \tau_{\alpha,-s} \varphi) + \widetilde{I}_1'.
\end{multline}
By Poincaré's inequality in $H_0^3(B_1)$, we have
\begin{equation}
	\label{eq:IR8-1}
	|\widetilde{I}_1'| \leq C  \|D^3 u\|_{L^2(B_{1})}\|D^3 \varphi\|_{L^2(B_{1})},
\end{equation}
where $C>0$ only depends on $\| \mathbb
Q\|_{L^\infty( \overline{B_{1}})}$. 

By using \eqref{eq:IR4-3},  \eqref{eq:IR4-4}, \eqref{eq:IR7-3}, \eqref{eq:IR8-1} in \eqref{eq:IR3-2}, we have, for every $\varphi \in H_0^3(B_1)$, 
\begin{equation}
	\label{eq:IR8-2}
	\int_{B_1} \mathbb Q D^3 (\tau_{\alpha,s} (\vartheta u)   ) \cdot D^3 \varphi =
	- \int_{B_{1}} \mathbb Q   (D^3 u) \cdot D^3( \vartheta \tau_{\alpha,-s} \varphi) + r_{\mathbb Q},
	\end{equation} 
with
\begin{equation}
	\label{eq:IR8-3}
	|r_{\mathbb Q}| \leq C  \|u\|_{H^3(B_{1})}\|D^3 \varphi\|_{L^2(B_{1})},
\end{equation}
where $C>0$ only depends on $\| \mathbb
Q\|_{C^{0,1}( \overline{B_{1}})}$.

By proceeding similarly with the term in $a(u,\varphi)$ containing the second order derivatives (with $u$ replaced by $\tau_{\alpha,s}(\vartheta u)$), for every $\varphi \in H_0^3(B_1)$ we have
\begin{equation}
	\label{eq:IR8-4}
	\int_{B_1} (\mathbb P + \mathbb P^h) D^2 (\tau_{\alpha,s} (\vartheta u)   ) \cdot D^2 \varphi =
	- \int_{B_{1}} (\mathbb P + \mathbb P^h)   (D^2 u) \cdot D^2( \vartheta \tau_{\alpha,-s} \varphi) + r_{\mathbb P},
\end{equation} 
with
\begin{equation}
	\label{eq:IR8-5}
	|r_{\mathbb P}| \leq C  \|u\|_{H^2(B_{1})}\|D^2 \varphi\|_{L^2(B_{1})},
\end{equation}
where $C>0$ only depends on $\| \mathbb
P\|_{C^{0,1}( \overline{B_{1}})}$, $\| \mathbb
P^h\|_{C^{0,1}( \overline{B_{1}})}$.

By \eqref{eq:IR8-2}--\eqref{eq:IR8-5}, for every $\varphi \in H_0^3(B_1)$ we have
\begin{equation}
	\label{eq:IR9-1}
	a(\tau_{\alpha,s} (\vartheta u) , \varphi  )= - a(u,  \vartheta \tau_{\alpha,-s} \varphi    ) + r,
\end{equation}
where, by Poincaré's inequality, 
\begin{equation}
	\label{eq:IR9-2}
	|r| \leq C  \|u\|_{H^3(B_{1})}\|D^3 \varphi\|_{L^2(B_{1})}
\end{equation}
and $C>0$ only depends on $\| \mathbb
P\|_{C^{0,1}( \overline{B_{1}})}$, $\| \mathbb
P^h\|_{C^{0,1}( \overline{B_{1}})}$, $\| \mathbb
Q\|_{C^{0,1}( \overline{B_{1}})}$. The function $\vartheta \tau_{\alpha,-s} \varphi  \in H_0^3(B_1)$ is a test function and then, by the weak formulation of the problem \eqref{eq:IR1-1}, for every $\varphi \in H_0^3(B_1)$ we have
\begin{equation}
	\label{eq:IR9-4}
	a(\tau_{\alpha,s} (\vartheta u) , \varphi  )
	\leq 
	C
	\| D^3 \varphi\|_{L^2(B_{1})} 	\| u\|_{H^3(B_{1})} ,
\end{equation}
where $C>0$ only depends on $\| \mathbb
P\|_{C^{0,1}( \overline{B_{1}})}$, $\| \mathbb
P^h\|_{C^{0,1}( \overline{B_{1}})}$, $\| \mathbb
Q\|_{C^{0,1}( \overline{B_{1}})}$. 

Let us take $\varphi = \tau_{\alpha,s} (\vartheta u)$. By the strong convexity of the strain energy (see Lemma \ref{Lemma:M6-1}), for every $s$ such that $|s| \leq \frac{1}{16}$ we have
\begin{equation}
	\label{eq:IR10-1}
	\| D^3 \tau_{\alpha,s} (\vartheta u)\|_{L^2(B_{1})}
	\leq 
	C
		\| u\|_{H^3(B_{1})}   
\end{equation}
and, therefore, 
\begin{equation}
	\label{eq:IR10-2}
	\| D^4 u\|_{L^2(B_{ \frac{1}{2}   })}
	\leq 
	C	\| u\|_{H^3(B_{1})} ,
\end{equation}
where $C>0$ only depends on $t$, $l$, $\xi_{\mathbb{Q}}$, $\| \mathbb
P\|_{C^{0,1}( \overline{B_{1}})}$, $\| \mathbb
P^h\|_{C^{0,1}( \overline{B_{1}})}$, $\| \mathbb
Q\|_{C^{0,1}( \overline{B_{1}})}$. 
\end{proof}

\begin{proof}[Proof of Theorem \ref{theo:loc-bndry-reg-res-exTheo8-4} ]
Let us assume for simplicity $\sigma=1$. Let us denote by $\vartheta \in C_0^{\infty}(\R^2)$ a function such that $0
\leq \vartheta(x) \leq 1$ for every $x \in \R^2$, $\vartheta
\equiv 1$ in $B_{\rho}$, $\vartheta \equiv 0 $ in $\R^2 \setminus
B_{\sigma_0}$, $|D^k \vartheta | \leq C$, $k=1,...,4$,
where $ \rho =
\frac{1}{2}$, $\sigma_0= \frac{1}{2}(\rho+1)=\frac{3}{4}$, and $C>0$ is an absolute constant.

For every function $\varphi \in H_{\Gamma_{1}^+}^2(B_{1}^+)$, we
still denote by $\varphi \in H^2(\R_2^+)$ its extension to
$\R_+^2$ obtained by assuming $\varphi \equiv 0$ in $\R_+^2
\setminus B_{1}^+$. Let $s \in \R \setminus \{0 \}$, with $|s|
\leq \frac{1}{16}$. Let us notice that if $u\in H^3(B_{1}^+)$,
then $\tau_{1,s}(\vartheta u) \in H_{\Gamma_{1}^+}^3(B_{1}^+)$.

 We shall firstly derive an estimate of the tangential derivative $\frac{\partial}{\partial x_1} D^3u$.

By using arguments analogous to those adopted to prove \eqref{eq:IR9-1}, \eqref{eq:IR9-2} in the study of the interior regularity, for every $\varphi \in H^3_{ \Gamma_1^+  }(B_1^+)$, we have
\begin{equation}
	\label{eq:BR-EeK}
	a_+(\tau_{1,s} (\vartheta u) , \varphi  )= - a_+(u,  \vartheta \tau_{1,-s} \varphi    ) + r_+,
\end{equation}
where, by Poincaré's inequality applied on $H^3_{ \Gamma_1^+  }(B_1^+)$, 
\begin{equation}
	\label{eq:BR9-remainder}
	|r_+| \leq C  \|u\|_{H^3(B_{1}^+)}\|D^3 \varphi\|_{L^2(B_{1}^+)},
\end{equation}
with a constant $C>0$ only depending on $E$ and $K$. The function $\vartheta \tau_{1,-s} \varphi   \in H^3_{ \Gamma_1^+  }(B_1^+)$ is a test function and then, by \eqref{eq:BR-EeK}, \eqref{eq:BR9-remainder}, by the weak formulation of the problem \eqref{eq:BR1-1} and by Poincaré's inequality in $H^3_{\Gamma_1^+}(B_1^+)$, for every $\varphi  \in H^3_{ \Gamma_1^+  }(B_1^+)$ we have
\begin{equation}
	\label{eq:BR12-3}
	|a_+(\tau_{1,s}(\vartheta u), \varphi) | 
	\leq 
	C 
	\left (
	G + \| u \|_{H^3 ( B_{1}^+)} 
	\right 	) \| D^3 \varphi
	\|_{L^2 ( B_{1}^+)} ,
\end{equation}
where $C>0$ only depends on $E$ and $K$.

We now estimate {}from below $a_+(\psi,\psi)$ for every $\psi \in H^3_{ \Gamma_1^+  }(B_1^+)$. We write  
\begin{equation}
	\label{eq:BR13-1}
	a_+(\psi, \psi) = \int_{B_1^+} \mathbb {K}^{(3,3)} D^3 \psi \cdot D^3 \psi+ \mathcal{R}(\psi,\psi), 
\end{equation}
where
\begin{equation}
	\label{eq:BR13-2}
	\mathcal{R}(\psi,\psi) = a_+^{\mathbb  E  }(\psi,\psi) + 
	\int_{B_1^+} \sum_{i,j=1; i+j<6}^3 \mathbb K^{(i,j)} D^i \psi \cdot D^j \psi.
\end{equation}
By Poincaré's inequality in $H^3_{ \Gamma_1^+  }(B_1^+)$ and by the standard inequality $2ab \leq \epsilon a^2 + \epsilon^{-1}b^2$ for every $a,b \in \R$ and $\epsilon >0$,  the remainder $\mathcal{R}(\psi,\psi)$ can be estimated as follows
\begin{equation}
	\label{eq:BR14-2}
	|\mathcal{R}(\psi,\psi)| 
	\leq
	C
	\left (
	\epsilon \|D^3 \psi \|_{L^2 (B_1^+)   }^2 +
	\left ( 1 + \frac{1}{\epsilon}  \right )
	 \|D^2 \psi \|_{L^2 (B_1^+)   }^2
	\right)
\end{equation}
for every $\psi \in H^3_{ \Gamma_1^+  }(B_1^+)$, where $C>0$ only depends on $E$ and $K$. Taking $\psi = \tau_{1,s}(\vartheta u)$, by the strong convexity of $\mathbb {K}^{(3,3)}$ and choosing $\epsilon$ small enough in \eqref{eq:BR14-2}, we have
\begin{equation}
	\label{eq:BR15-1}
	a_+(\tau_{1,s}(\vartheta u), \tau_{1,s}(\vartheta u)) 
	\geq
	C_2 
	\|D^3 ( \tau_{1,s}(\vartheta u)   ) \|_{L^2 (B_1^+)   }^2
	-
	C_3
	\|D^2 ( \tau_{1,s}(\vartheta u)   ) \|_{L^2 (B_1^+)   }^2
\end{equation}
where $C_2>0$, $C_3 >0$ only depend on $\xi_{   \mathbb K}$, $E$, $K$. By \eqref{eq:BR12-3} (with $\varphi = \tau_{1,s}(\vartheta u)$) and \eqref{eq:BR15-1}, for every $s$ such that $|s| \leq 1/16$, we have
\begin{eqnarray}
	\label{eq:BR16-1}
	\|D^3 ( \tau_{1,s}(\vartheta u)   ) \|_{L^2 (B_1^+)   }^2
	\leq C
	\|D^3 ( \tau_{1,s}(\vartheta u)   ) \|_{L^2 (B_1^+)   }
	\left (
	G + \| u \|_{H^3 ( B_{1}^+)} 
	\right 	)
	+
	C
	\|u \|_{H^3 (B_1^+)   }^2,
\end{eqnarray}
which implies 
\begin{equation}
	\label{eq:BR17-2}
	\left\| \frac{\partial }{\partial x_1}  D^3 u \right\|_{L^2 ( B_{\rho}^+)}
	\leq
	C \left (G + \| u \|_{H^3 ( B_{1}^+)} \right ),
\end{equation}
where $C>0$ only depends on $\xi_{ \mathbb{K}  }$, $E$, $K$.

To obtain an analogous estimate for the normal derivative $\frac{\partial}{\partial x_2} D^3 u$, we use the
following lemma.

\begin{lem} (\cite[Lemma $9.3$]{l:agmon})
	\label{lem:agmon}
	Assume that the function $v \in L^2 (B_{\sigma}^+)$ has weak
	tangential derivative $ \frac{\partial v}{\partial x_1} \in L^2
	(B_{\sigma}^+)$ and there exists a constant $K_0>0$ such
	that
	\begin{equation}
		\label{eq:agmon-1}
		\left |\int_{B_{\sigma}^+} v \frac{\partial^3 \varphi}{\partial
			x_2^3} \right |
		\leq K_0 \| \varphi \|_{H^2 ( B_{\sigma}^+)}, \qquad \hbox{for
			every } \varphi \in C_0^{\infty}(B_{\sigma}^+).
	\end{equation}
	Then, for every $\rho < \sigma$, $v \in H^1(B_{\rho}^+)$ and
	\begin{equation}
		\label{eq:agmon-2}
		\| v \|_{H^1 ( B_{\rho}^+)}
		\leq
		C
		\left (
		K_0 + \| v \|_{L^2 ( B_{\sigma}^+)} + \sigma\left\|\frac{\partial v}{\partial x_1} \right\|_{L^2 (
			B_{\sigma}^+)}
		\right ),
	\end{equation}
	where $C>0$ only depends on $\frac{\rho}{\sigma}$.
\end{lem}
In what follows we shall consider test functions $\varphi \in C_0^\infty ( B_{\sigma_0}^+  )$, where $\sigma_0= 3/4$ (and $\sigma=1$). {}From the expression \eqref{eq:BR1-2} of $a_+(u,\varphi)$ we have
\begin{multline}
	\label{eq:BR18-1}
	\int_{B_{\sigma_0}^+} K_{222lmn}^{(3,3)} u_{,lmn}\varphi_{,222}=
	a_+(u, \varphi)
	-
	\int_{B_{\sigma_0}^+}
	\sum_{(i,j,k) \neq (2,2,2)}
	K_{ijklmn}^{(3,3)} u_{,lmn}\varphi_{,ijk}
	-
	\\
	-
		\int_{B_{\sigma_0}^+} \sum_{i,j=1}^2 \mathbb E^{(i,j)} D^i u \cdot D^j \varphi
		-
		\int_{B_1^+} \sum_{(i,j) \neq (3,3)} \mathbb K^{(i,j)} D^i u \cdot D^j \varphi.
\end{multline}
Let us estimate the terms on the right-hand side of \eqref{eq:BR18-1}. By the weak formulation \eqref{eq:BR1-1}, we have
\begin{equation}
	\label{eq:BR18-2}
	|a_+(u, \varphi)| \leq G \| \varphi\|_{ H^2 (B_{\sigma_0}^+)   }.
\end{equation}
Let us consider the second term on the right hand side of \eqref{eq:BR18-1}. Since $(i,j,k) \neq (2,2,2)$, without loss of generality we can assume $k=1$. Integrating by parts with respect to $x_1$, for every $\varphi \in C_0^\infty ( B_{\sigma_0}^+  )$ we have
\begin{multline}
	\label{eq:BR19-1}
	\int_{B_{\sigma_0}^+}
	K_{ij1lmn}^{(3,3)} u_{,lmn}\varphi_{,ij1} = 
	-
	\int_{B_{\sigma_0}^+}
	(K_{ij1lmn}^{(3,3)} u_{,lmn})_{,1}\varphi_{,ij}
	=
	\\
	=
	-
	\int_{B_{\sigma_0}^+}
	K_{ij1lmn,1}^{(3,3)} u_{,lmn}\varphi_{,ij}
	-
	K_{ij1lmn}^{(3,3)} u_{,lmn1}\varphi_{,ij}
\end{multline}
and therefore, by \eqref{eq:BR17-2} and H\"{o}lder's inequality, we have
\begin{equation}
	\label{eq:BR19-2}
	\left |
	\int_{B_{\sigma_0}^+}
	K_{ij1lmn}^{(3,3)} u_{,lmn}\varphi_{,ij1}
	\right |
	\leq 
	C
	(
	G + \| u\|_{H^3(B_1^+)}
	)
	\| D^2 \varphi\|_{L^2(B_{\sigma_0}^+)},
\end{equation}
where $C>0$ only depends on $\xi_{\mathbb K}$, $E$, $K$.

By using Poincaré's inequality in $C^\infty_0 (B_{\sigma_0}^+) $,  the terms involving $\mathbb E^{(i,j)}$ in \eqref{eq:BR18-1} can be estimated as follows
\begin{equation}
	\label{eq:BR19-3}
	\left |
	\int_{B_{\sigma_0}^+} \sum_{i,j=1}^2 \mathbb E^{(i,j)} D^i u \cdot D^j \varphi
	\right |
	\leq 
	C
	\| u\|_{H^2(B_1^+)} 
	\| D^2 \varphi\|_{L^2(B_{\sigma_0}^+)},
\end{equation}
where $C>0$ only depends on $E$. The estimate of the terms in \eqref{eq:BR18-1} involving $\mathbb K^{(i,j)}$ is easy, with the exception of those terms which involve $D^3 \varphi$ (that is $j=3$). These integrals can be estimated by integrating by parts and descharging one derivative {}from $\varphi$ to $u$. Let us consider the term with $\mathbb K^{(2,3)}$, the analysis of the term with $\mathbb K^{(1,3)}$ being similar. Passing to Cartesian coordinates and integrating by parts, for every $\varphi \in C^\infty_0 (B_{\sigma_0}^+) $ we have
\begin{multline}
	\label{eq:BR20-1}
	\left |
	\int_{B_{\sigma_0}^+}
	\mathbb K^{(2,3)} D^2 u \cdot D^3 \varphi 
	\right |
	=
	\left |
	\int_{B_{\sigma_0}^+}
	K_{ijklm}^{(2,3)} u_{,lm} \varphi_{,ijk}
	\right | = 
	\\
	=
	\left |
	\int_{B_{\sigma_0}^+}
	(K_{ijklm}^{(2,3)} u_{,lm})_{,k}\varphi_{,ij}
	\right |
	\leq
	C 	\| u\|_{H^3(B_1^+)} 
	\| D^2 \varphi\|_{L^2(B_{\sigma_0}^+)},
\end{multline}
where $C>0$ only depends on $K$. Therefore, by using \eqref{eq:BR18-2}, \eqref{eq:BR19-2}--\eqref{eq:BR20-1} in \eqref{eq:BR18-1}, for every $\varphi \in C^\infty_0 (B_{\sigma_0}^+) $ we obtain
\begin{equation}
	\label{eq:BR20-2}
	\left |
	\int_{B_{\sigma_0}^+} K_{222lmn}^{(3,3)} u_{,lmn}\varphi_{,222}
	\right |
	\leq 
	C
	(
	G +\| u\|_{H^3(B_1^+)} 
	)
	\| \varphi\|_{H^2(B_{\sigma_0}^+)},
\end{equation}
where $C>0$ only depends on $\xi_{\mathbb K}$, $E$, $K$.

Let us define 
\begin{equation}
	\label{eq:BR21-1}
	v = \sum_{l,m,n=1}^2 K_{222lmn}^{(3,3)}  u_{,lmn}.
\end{equation}
By Lemma \ref{lem:agmon}, for every $\rho < \sigma_0$, the function $v$ belongs to $H^1(B_\rho^+)$ and, by \eqref{eq:BR17-2}, 
\begin{equation}
	\label{eq:BR21-2}
	\| v \|_{H^1 ( B_{\rho}^+)} \leq C \left (G + \| u \|_{H^3 ( B_{1}^+)} \right ),
\end{equation}
where $C>0$ only depends on $\xi_{\mathbb K}$, $E$, $K$. By the ellipticity of the tensor $\mathbb K^{(3,3)}$, $K_{222222}^{(3,3)}>0$ and then
\begin{equation}
	\label{eq:BR21-3}
	u_{,222} = ( K_{222222}^{(3,3)}  )^{-1}
	\left (
	v - \sum_{(l,m,n)\neq (2,2,2)} K_{222lmn}^{(3,3)} u_{,lmn}
	\right ).
\end{equation}
By \eqref{eq:BR17-2}, we deduce that $u_{,222} \in H^1( B_{\rho}^+)$, with
\begin{equation}
	\label{eq:BR21-4}
	\|u,_{2222}\|_{L^2 ( B_{\rho}^+)}
	\leq
	C \left (G + \| u \|_{H^3 ( B_{1}^+)} \right ),
\end{equation}
where $C>0$ only depends on $\xi_{\mathbb K}$, $E$, $K$. Finally, by \eqref{eq:BR17-2} and \eqref{eq:BR21-4} we obtain the wished inequality
\begin{equation}
	\label{eq:BR21-5}
	\|u\|_{H^4 ( B_{\rho}^+)}
	\leq C \left (G + \| u \|_{H^3 ( B_{1}^+)} \right ),
\end{equation}
where $C>0$ only depends on $\xi_{\mathbb K}$, $E$, $K$.
\end{proof}

We conclude this section with an improved  interior regularity result, which will be useful in dealing with the unique continuation properties obtained in Section \ref{sec:Doubl-Three-Sphere}.

\begin{theo} [{\bf Improved interior regularity}]
	\label{theo:loc-int-reg-res-improved}
	Let $B_{\sigma}$ be an open ball in $\R^2$ centered
	at the origin and with radius $\sigma$. Let $u \in H^3(B_{\sigma})$ be such that
	\begin{equation}
		\label{eq:HIR1-1}
		a(u, \varphi) = 0		   \qquad \hbox{for every }
		\varphi \in H^3_0(B_{\sigma}),
	\end{equation}
	with
	\begin{equation}
		\label{eq:HIR1-2}
		a(u, \varphi) = \int_{B_{\sigma}}  (\mathbb P + \mathbb P^h) D^2u \cdot D^2 \varphi + \mathbb Q D^3 u \cdot D^3 \varphi,
	\end{equation}
	where the tensors $\mathbb P, \mathbb P^h \in C^{1,1} (\overline{B_{\sigma}}, \mathcal{L}(\widehat{\M}^2, \widehat{\M}^2   )    )$, $\mathbb Q \in C^{2,1} (\overline{B_{\sigma}}, \mathcal{L}(\widehat{\M}^3, \widehat{\M}^3   )  )$ defined in \eqref{eq:M3-1}, \eqref{eq:M3-2}, \eqref{eq:M4-4} satisfy the strong convexity conditions \eqref{eq:M6-1}, \eqref{eq:M6-2}, respectively. 
	
	Then, $u \in H^6( B_{  \frac{\sigma}{8}  })$ and we have
	\begin{equation}
		\label{eq:HIR1-3}
		\|u\|_{ H^6 (B_{\frac{\sigma}{8}})} \leq C  \|u\|_{ H^3 (B_{\sigma})},
	\end{equation}
	where $C>0$ only depends on $\frac{t}{r_0}$, $\frac{l}{r_0}$, $\xi_{\mathbb Q}$, $\xi_{\mathbb P}$, $\|\mathbb P\|_{C^{1,1}( \overline{B_\sigma})}$, $\|\mathbb P^h\|_{C^{1,1}( \overline{B_\sigma})}$, $\|\mathbb Q\|_{C^{2,1}( \overline{B_\sigma})}$.
\end{theo}
\begin{proof}
We can assume, without loss of generality, $\sigma =1$. By Theorem  \ref{theo:loc-int-reg-res-exTheo8-3} we know that $u \in H^4(B_{ \frac{1}{2} }   )$. Therefore, differentiating \eqref{eq:M2-1} with respect to $x_p$, $p=1,2$, and integrating by parts, we obtain
\begin{equation}
	\label{eq:HIR2-1}
	a(u_{,p}, \varphi) = \textit{l}	_p(\varphi)	   \qquad \hbox{for every }
	\varphi \in H^3_0(B_{ \frac{1}{2}   }),
\end{equation}
where $a(\cdot, \cdot)$ is defined in \eqref{eq:HIR1-2} and 
\begin{equation}
	\label{eq:HIR2-2}
	\textit{l}	_p(\varphi)=
	\int_{ B_{ \frac{1}{2}   } } 
	(P_{ijlm}  +  P_{ijlm}^h  )_{,p} u_{,lm} \varphi_{,ij} -
	( Q_{ijklmn,p} u_{,lmn}   )_{,i} \varphi_{,jk}.	
\end{equation}
By the regularity assumptions on the coefficients and \eqref{eq:IR1-4}, we have
\begin{equation}
	\label{eq:HIR2-3}
	|\textit{l}	_p(\varphi)| \leq 
	C \| u \|_{H^3(B_1)} \| D^2 \varphi \|_{L^2 (B_{\frac{1}{2}})  },
\end{equation}
where $C>0$ is a constant only depending on $t$, $l$, $\xi_{\mathbb Q}$, $\xi_{\mathbb P}$, $\|\mathbb P\|_{C^{0,1}( \overline{B_1})}$, $\|\mathbb P^h\|_{C^{0,1}( \overline{B_1})}$, $\|\mathbb Q\|_{C^{1,1}( \overline{B_1})}$.

At this point, by \eqref{eq:HIR2-1} and \eqref{eq:HIR2-3}, we can use arguments analogous to those adopted in the proof of \eqref{eq:IR9-1} (with $u$ replaced by $u_{,p}$) to obtain 
\begin{equation}
	\label{eq:HIR2-4}
	\|D^5 u\|_{ L^2 (B_{\frac{1}{4}})} \leq C  \|u\|_{ H^3 (B_{1})}.
\end{equation}
Finally, estimate \eqref{eq:HIR1-3} follows by iterating once more the above procedure.
\end{proof}

\begin{rem}
   \label{rem:estensione anisotropo}
	Let us notice that, as it is evident {}from the proofs, Proposition \ref{prop:DirectProblem}, Theorem \ref{prop:glo-reg-res-exProp8-2} and Theorem \ref{theo:loc-int-reg-res-improved} extend to the anisotropic case, for $\mathbb{P}$, $\mathbb{P}^h$, $\mathbb{Q}$ satisfying the symmetry conditions \eqref{eq:M3-3}, \eqref{eq:M5-3}, and the strong convexity conditions \eqref{eq:M6-1}, \eqref{eq:M6-2}.

\end{rem}

\section{Doubling and Three spheres inequalities}\label{DTSI}

In this section we derive unique continuation results in the form of three-spheres and doubling inequalities for solutions to the differential inequality
\begin{eqnarray}\label{eq:diffineq}
|\Delta^3 u| \le M(|D\Delta^2 u| + \sum_{k=0}^{4} |D^k u|) \ \ \ \mbox{in} \ \ B_1 
\end{eqnarray}
where $M$ is a positive constant. Let us notice that the solutions to \eqref{eq:M2-1}, with $\mathbb{P},\mathbb{P}^h$ and $\mathbb{Q}$ given by \eqref{eq:M3-1}, \eqref{eq:M3-2} and \eqref{eq:M4-4}
respectively, satisfy \eqref{eq:diffineq}; see Lemma \ref{diffineqlemma} for a precise statement. Our method is based on Carlemann estimates. 
 


\subsection{Carleman estimates}\label{sec:Doubl-Three-Sphere}

We shall need the following results, see \cite{mrv:mat-cat}, Proposition 5.1 and inequality (5.46)  in the proof of Proposition 3.5, respectively.

\begin{prop}[{\bf Carleman estimate for $\Delta$}]\label{prop:Carlm-delta}

Let $\epsilon\in\left(0,\frac{1}{2}\right]$. Let us define
\begin{equation}
    \label{eq:24.2}
       \rho(x) = \phi_{\epsilon}\left(|x|\right), \mbox{ for } x\in B_1\setminus
\{0\},
\end{equation}
where
\begin{equation}
    \label{eq:24.3}
        \phi_{\epsilon}(s) = \frac{s}{\left(1+s^{\epsilon}\right)^{1/\epsilon}}.
\end{equation}
Then there exist $\tau_0>1$, $C>1$, only depending on $\epsilon$,
such that for every $\tau\geq \tau_0$ and for every $u\in
C^\infty_0(B_1\setminus \{0\})$
\begin{equation}
    \label{Carlm-delta}
    C\int\rho^{4-2\tau}|\Delta u|^2dx\geq\sum_{k=0}^1
\tau^{3-2k}\int\rho^{2k+\epsilon-2\tau}|D^ku|^2dx.
\end{equation}
Furthermore we have
\begin{equation}
    \label{Carlm-delta-doub}
    C\int\rho^{4-2\tau}|\Delta u|^2dx\geq
\tau^2r\int\rho^{-1-2\tau}u^2dx+\\+\sum_{k=0}^1
\tau^{3-2k}\int\rho^{2k+\epsilon-2\tau}|D^ku|^2dx
\end{equation}
for every $\tau\geq \tau_0$, for every $r\in (0, 1)$ and for every
$u\in C^\infty_0(B_{1}\setminus\ \overline{B}_{r/4})$.
\end{prop}

\begin{rem}
   \label{rem:stima_rho}
    Let us notice that
    \begin{equation*}
\frac{s}{2^{1/\epsilon}}\leq \phi_{\epsilon}\leq s, \quad \forall s, 0\leq s\leq 1,
\end{equation*}
\begin{equation}
    \label{eq:stima_rho}
    \frac{|x|}{2^{1/\epsilon}}\leq \rho(x)\leq |x|, \quad \forall x\in B_1.
\end{equation}

\end{rem}

\bigskip

\begin{prop} [{\bf Carleman estimate for $\Delta^2$}]
    \label{prop:Carleman}
Let $\epsilon\in\left(0,\frac{1}{2}\right)$. Let $\rho$ and
$\phi_{\epsilon}$ the same functions defined in Proposition
\ref{prop:Carlm-delta}. Then there exist absolute constants
$\overline{\tau}>1$, $C>1$ depending on $\varepsilon$ only, such
that

\begin{equation}
    \label{eq:24.4}
    \sum_{k=0}^3\tau^{6-2k}\int \rho^{2k+2\epsilon-2\tau}|D^k
U|^2dx  \leq C
    \int\rho^{8-2\tau}(\Delta^2 U)^2dx,
\end{equation}
for every $\tau\geq \overline{\tau}$ and
for every $U\in C^\infty_0(B_1\setminus \{0\})$.
\end{prop}

\bigskip

\begin{lem}\label{formule}
Given $\zeta\in C^2(B_1\setminus\{0\})$ and $u\in
C^\infty_0(B_1\setminus\{0\})$, the following identities hold
true:
\begin{subequations}
\label{identity}
\begin{equation}
\label{identity1} \int \zeta u\Delta u=-\int(\zeta|D u|^2+(D
u\cdot D \zeta)u),
\end{equation}
\begin{equation}
\label{identity2} \int \zeta\sum_{j,k=1}^2|\partial_{jk}u|^2=\int(-D^2\zeta D u\cdot
D u+\Delta \zeta|D u|^2+\zeta(\Delta u)^2),
\end{equation}
\begin{gather}
\label{identity3}
\int \zeta\sum_{i,j,k=1}^2|\partial_{ijk}u|^2=-\int \zeta\Delta u \Delta^2
u+\\\nonumber
+\int(-\hbox{tr}(D^2 uD^2 \zeta D^2u)
+\Delta \zeta|D^2 u|^2+\frac{1}{2}\Delta \zeta(\Delta u)^2).
\end{gather}
\end{subequations}
\end{lem}

For a proof of the above identities see \cite[pp. 2351--2352]{MRV2007}.

\begin{prop} [{\bf Carleman estimate for $\Delta^3$}]
    \label{prop:Carleman 3Laplace}
Let $\epsilon\in\left(0,\frac{1}{5}\right]$. Let $\rho$ and
$\phi_{\epsilon}$ the same functions defined in Proposition
\ref{prop:Carlm-delta}.

Then there exist constants $\overline{\tau}>1$, $C>1$ and
$R_1\in (0,1]$ only depending on $\epsilon$ such that
\begin{equation}
 \begin{aligned}\label{Carleman0-3L}
C\int\rho^{4-2\tau}|\Delta^3 u|^2dx\geq
\tau\int\rho^{2+\epsilon-2\tau}\left|D\Delta^2u\right|^2dx+
 \sum_{k=0}^4\tau^{9-2k}\int \rho^{2k+5\epsilon-8-2\tau}|D^k
u|^2dx,
\end{aligned}
\end{equation}
for every $\tau\geq \overline{\tau}$ and for every $u\in
C^\infty_0(B_{R_1}\setminus\ \{0\})$.

Furthermore we have
\begin{equation}
 \begin{aligned}
    \label{eq:24.4-3L}
    C\int\rho^{4-2\tau}|\Delta^3 u|^2dx&\geq
\tau\int\rho^{2+\epsilon-2\tau}\left|D\Delta^2u\right|^2dx+
 \sum_{k=0}^4\tau^{9-2k}\int \rho^{2k+5\epsilon-8-2\tau}|D^k
u|^2dx+\\&+\tau^6r^3\int\rho^{-11-2\tau}u^2dx,
\end{aligned}
\end{equation}
\noindent for every $\tau\geq \overline{\tau}$, for every $r\in (0,
R_1)$ and for every $u\in C^\infty_0(B_{R_1}\setminus\
\overline{B}_{r/4})$.
\end{prop}

\begin{proof}
Let us apply estimate \eqref{Carlm-delta} to $\Delta^3
u=\Delta\left(\Delta^2u\right)$ to obtain (for brevity we omit $dx$
in the integrals)
\begin{equation}
 \begin{aligned}\label{2-1-3L}
C\int\rho^{4-2\tau}|\Delta^3 u|^2&\geq
\tau\int\rho^{2+\epsilon-2\tau}\left| D\Delta^2u \right|^2+\tau^3\int\rho^{\epsilon-2\tau}\left|\Delta^2u\right|^2,
\end{aligned}
\end{equation}
for every $\tau\geq \tau_0$ and for every $u\in
C^\infty_0(B_1\setminus \{0\})$, where $C$ depends on $\epsilon$
only. Now, in order to estimate {}from below the second term at
right hand side of \eqref{2-1-3L}, we apply estimate \eqref{eq:24.4}.
To do this, we change $\tau$ in $4-\frac{\varepsilon}{2}+\tau$ in
\eqref{eq:24.4}. We have

\begin{equation}\label{3-1-3L}
C\int\rho^{\epsilon-2\tau}|\Delta^2 u|^2\geq
\sum_{k=0}^3\tau^{6-2k}\int \rho^{2k+3\epsilon-8-2\tau}|D^k u|^2,
\end{equation}
for every $\tau\geq \overline{\tau}$ and for every $u\in
C^\infty_0(B_1\setminus \{0\})$. By \eqref{2-1-3L} and
\eqref{3-1-3L} we get
\begin{equation}\label{1-2-3L}
C\int\rho^{4-2\tau}|\Delta^3 u|^2\geq
\tau\int\rho^{2+\epsilon-2\tau}\left|D\Delta^2u\right|^2+\frac{\tau^3}{2}\int\rho^{\epsilon-2\tau}\left|\Delta^2u\right|^2+
+\frac{1}{2}\sum_{k=0}^3\tau^{9-2k}\int
\rho^{2k+3\epsilon-8-2\tau}|D^k u|^2.
\end{equation}
Now we need to estimate {}from below the second term on the right hand
side of \eqref{1-2-3L}. By \eqref{identity2} (just writing $\Delta
u$ instead of $u$) we have
\begin{equation}
 \begin{aligned}\label{2-2-3L}
 \int \zeta(\Delta^2 u)^2 &=\int  \zeta|D^2\Delta u|^2 -\int\left(\Delta \zeta|D \Delta u|^2-D^2\zeta D \Delta u\cdot D
\Delta u\right).
\end{aligned}
\end{equation}
Furthermore let $\sigma\geq \epsilon$ be a number that we will
choose later. Since $\rho \leq 1$ we have trivially
\begin{equation}\label{1-3-3L}
\int\rho^{\epsilon-2\tau}\left|\Delta^2u\right|^2\geq
\int\rho^{\sigma-2\tau}\left|\Delta^2u\right|^2.
\end{equation}
Now, by choosing $$\zeta=\rho^{\sigma-2\tau},$$ we have for
$\tau\geq 1$

$$\left|D^2\zeta\right|\leq C\tau^2\rho^{-2+\sigma-2\tau},$$
where $C>0$ only depends on $\epsilon$.

Hence, by \eqref{2-2-3L} and \eqref{1-3-3L} we have
\begin{equation}
 \begin{aligned}\label{2-3-3L}
 \int\rho^{\epsilon-2\tau}\left|\Delta^2u\right|^2 &\geq \int
 \rho^{\sigma-2\tau}\left|D^2\Delta u\right|^2-C_{\ast}\tau^2\int
 \rho^{-2+\sigma-2\tau}\left|D^3 u\right|^2,
\end{aligned}
\end{equation}
where $C_{\ast}$ only depends on $\epsilon$. By \eqref{2-3-3L} we
have trivially
\begin{equation}\label{2n-3-3L}
 \tau^3\int\rho^{\epsilon-2\tau}\left|\Delta^2u\right|^2 \geq
 \tau\int\rho^{\epsilon-2\tau}\left|\Delta^2u\right|^2\geq
 \tau\int
 \rho^{\sigma-2\tau}\left|D^2\Delta u\right|^2-C_{\ast}\tau^3\int
 \rho^{-2+\sigma-2\tau}\left|D^3 u\right|^2.
\end{equation}
Now, by \eqref{1-2-3L} and \eqref{2n-3-3L} we have
\begin{eqnarray}\label{3-3-3L}
&&C\int\rho^{4-2\tau}|\Delta^3 u|^2\geq
\tau\int\rho^{2+\epsilon-2\tau}\left|D\Delta^2u\right|^2+\tau \int
\rho^{\sigma-2\tau}\left|D^2\Delta u\right|^2+\\
&&+\underset{J
}{\underbrace{\left(-C_{\ast}\tau^3\int
 \rho^{-2+\sigma-2\tau}\left|D^3
 u\right|^2+\tau^3\int\rho^{-2+3\epsilon-2\tau}\left|D^3u\right|^2\right)}}
 +\sum_{k=0}^2\tau^{9-2k}\int \rho^{2k+3\epsilon-8-2\tau}|D^k \nonumber
u|^2dx,
\end{eqnarray}
for every $\tau\geq \overline{\tau}$ and for every $u\in
C^\infty_0(B_1\setminus \{0\})$. Now, let us choose
$$\sigma=4\epsilon$$ and denote
$$R_0=\left(1/2C_{\ast}\right)^{1/\epsilon}.$$ Taking into account \eqref{eq:stima_rho}, we have, for every $u\in
C^\infty_0(B_{R_0}\setminus \{0\})$,
\begin{equation}\label{3n-3-3L}
J=\tau^3\int\left(-C_{\ast}\rho^{\epsilon}+1\right)\rho^{-2+3\epsilon-2\tau}
\left|D^3u\right|^2\geq
\frac{\tau^3}{2}\int\rho^{-2+3\epsilon-2\tau}\left|D^3u\right|^2.
\end{equation}
By \eqref{3-3-3L} and \eqref{3n-3-3L} we get
\begin{equation}\label{2-5-3L}
C\int\rho^{4-2\tau}|\Delta^3 u|^2\geq
\tau\int\rho^{2+\epsilon-2\tau}\left|D\Delta^2u\right|^2+\tau \int
\rho^{4\epsilon-2\tau}\left|D^2\Delta u\right|^2
 +\sum_{k=0}^3\tau^{9-2k}\int \rho^{2k+3\epsilon-8-2\tau}|D^k
u|^2,
\end{equation}
for every $\tau\geq \overline{\tau}$ and for every $u\in
C^\infty_0(B_{R_0}\setminus \{0\})$.

In order to obtain the term with $\left|D^4u\right|^2$ we use again
\eqref{identity2} in the following form (just writing
$\partial^2_{hk} u$, $h,k=1,2$  instead of $u$)

\begin{equation}
\label{identity2-n} \int \zeta(\Delta
\partial^2_{hk} u)^2=\int\zeta\left|D^2\partial^2_{hk} u\right|^2-\int \left(\Delta \zeta\left|D
\partial^2_{hk} u\right|^2-D^2\zeta D \partial^2_{hk} u\cdot D \partial^2_{hk} u \right)
\end{equation}
for $h,k=1,2$. By choosing
$$\zeta=\rho^{5\epsilon-2\tau}$$
we have
$$\left|D^2\zeta\right|\leq C\tau^2\rho^{-2+5\epsilon-2\tau},$$
where $C>0$ only depends on $\epsilon$. Hence, by \eqref{identity2-n}
we have
\begin{equation}\label{pag8}
 \tau\int \rho^{4\epsilon-2\tau} \left|D^2\Delta
u\right|^2\geq \tau\int \rho^{5\epsilon-2\tau} \left|D^2\Delta
u\right|^2\geq \tau\int \rho^{5\epsilon-2\tau} \left|D^4
u\right|^2-\widetilde{C}\tau^3 \int \rho^{-2+5\epsilon-2\tau}
\left|D^3u\right|^2.
\end{equation}
By \eqref{2-5-3L} and \eqref{pag8} we have that 
\begin{eqnarray}\label{1-8-3L}
&&C\int\rho^{4-2\tau}|\Delta^3 u|^2\geq
\tau\int\rho^{2+\epsilon-2\tau}\left|D\Delta^2u\right|^2+\tau \int
\rho^{5\epsilon-2\tau}\left|D^4
u\right|^2+\nonumber\\
&&\underset{\widetilde{J}
}{\underbrace{\tau^3\int\left(-\widetilde{C}\rho^{2\epsilon}+1\right)\rho^{-2+3\epsilon-2\tau}\left|D^3u\right|^2}}
 +\sum_{k=0}^2\tau^{9-2k}\int \rho^{2k+3\epsilon-8-2\tau}|D^k
u|^2.
\end{eqnarray}
Let
$$R_1=\min\left\{R_0,\left(1/2\widetilde{C}\right)^{1/2\epsilon}\right\}.$$
Taking into account \eqref{eq:stima_rho}, we have, for every $u\in
C^\infty_0(B_{R_1}\setminus \{0\})$,
\begin{equation}\label{1n-8-3L}
\widetilde{J}=\tau^3\int\left(-\widetilde{C}\rho^{2\epsilon}+1\right)\rho^{-2+3\epsilon-2\tau}
\left|D^3u\right|^2\geq
\frac{\tau^3}{2}\int\rho^{-2+3\epsilon-2\tau}\left|D^3u\right|^2.
\end{equation}
By \eqref{1-8-3L} and \eqref{1n-8-3L} we get
\begin{equation}
 \begin{aligned}\label{1nn-8-3L}
C\int\rho^{4-2\tau}|\Delta^3 u|^2\geq
\tau\int\rho^{2+\epsilon-2\tau}\left|D\Delta^2u\right|^2+
 \sum_{k=0}^4\tau^{9-2k}\int \rho^{2k+5\epsilon-8-2\tau}|D^k
u|^2,
\end{aligned}
\end{equation}
for every $\tau\geq \overline{\tau}$ and for every $u\in
C^\infty_0(B_{R_1}\setminus \{0\})$. Hence \eqref{Carleman0-3L} is
proved.

Now we prove \eqref{eq:24.4-3L}. Let us apply
\eqref{Carlm-delta-doub} in the following form

\begin{equation}\label{Carlm-delta-doub-1}
    \tau^2r\int\rho^{-1-2\tau}v^2dx\leq C\int\rho^{4-2\tau}|\Delta v|^2,
\end{equation}
for every $v\in C^\infty_0(B_{R_1}\setminus\ \overline{B}_{r/4})$.

If $v=\Delta^2u$ then \eqref{Carlm-delta-doub-1} gives
\begin{equation}\label{Carlm-delta-doub-2}
\int\rho^{4-2\tau}|\Delta^3 u|^2dx\geq
C^{-1}\tau^2r\int\rho^{-1-2\tau}|\Delta^2 u|^2.
\end{equation}
If $v=\Delta u$ then \eqref{Carlm-delta-doub-1} gives
\begin{eqnarray}\label{Carlm-delta-doub-3}
\int\rho^{-1-2\tau}|\Delta^2
u|^2&=&\int\rho^{4-2\left(\frac{5}{2}+\tau\right)}|\Delta^2
u|^2\geq C^{-1}\tau^2r\int\rho^{-6-2\tau}|\Delta
u|^2=C^{-1}\tau^2r\int\rho^{4-2(5+\tau)}|\Delta u|^2\nonumber\\
&\geq
&C^{-2}\tau^4r^2\int\rho^{-11-2\tau}u^2.
\end{eqnarray}
Hence, by \eqref{Carlm-delta-doub-2} and \eqref{Carlm-delta-doub-3}
we get
\begin{equation}\label{Carlm-delta-doub-4}
\int\rho^{4-2\tau}|\Delta^3 u|^2\geq
C^{-3}\tau^6r^3\int\rho^{-11-2\tau}u^2dx,
\end{equation}
for every $u\in C^\infty_0(B_{R_1}\setminus\ \overline{B}_{r/4})$ and $\tau\geq
\overline{\tau}$.

Finally, by \eqref{1nn-8-3L} and \eqref{Carlm-delta-doub-4} we obtain
\eqref{eq:24.4-3L}.
\end{proof}

\subsection{Doubling and Three Sphere Inequalities}\label{doubling}

\begin{lem}\label{diffineqlemma}

Let $\mathbb{P}, \mathbb{P}^h \in C^{1,1}(\overline{B_1}, {\cal L} (\widehat{\mathbb{M}}^2, \widehat{\mathbb{M}}^2)), \mathbb{Q} \in C^{2,1}(\overline{B_1}, {\cal L} (\widehat{\mathbb{M}}^3,\widehat{\mathbb{M}}^3))$ be given by \eqref{eq:M3-1}, \eqref{eq:M3-2}, \eqref{eq:M4-4} and satisfying the strong convexity conditions  \eqref{eq:M6-1}, \eqref{eq:M6-2},  respectively. 

Let $u\in H^6(B_1)$ be a weak solution to \eqref{eq:M2-1}. Then, there exists a constant $M>0$
 depending on $M_2, \alpha_0, t, l$ only, such that 
 \begin{eqnarray}\label{eq:diffineq2}
|\Delta^3 u| \le M(|D\Delta^2 u| + \sum_{k=0}^{4} |D^k u|) \ \ \ \mbox{in} \ \ B_1 \ , 
\end{eqnarray}
where $M_2=\|\mathbb{P} \|_{C^{1,1}(\overline{B_1})} + \|\mathbb{P}^h \|_{C^{1,1}(\overline{B_1})} + \|\mathbb{Q} \|_{C^{2,1}(\overline{B_1})}$ .

\end{lem}
\begin{proof}
The proof follows by a differentiating argument and formulas \eqref{eq:M5-2}.
\end{proof}

\begin{lem}[Caccioppoli-type inequality]
    \label{lem:intermezzo-cube}
        Let $K$ be a positive number and let us assume that $u\in H^6(B_1)$ satisfies the inequality
\begin{equation}
    \label{Cacciop-A1-1}
    \left|\Delta^3 u\right|\leq K\sum_{k=0}^5\left|D^k
    u\right|.
\end{equation}
Then, for every $r$, $0<r<1$, we have
        \begin{equation}
    \label{eq:12a.2-cube}
    \|D^hu\|_{L^2(B_{\frac{r}{2}})}\leq \frac{C}{r^h}\|u\|_{L^2(B_r)}, \quad \forall
    h=1, ..., 6,
\end{equation}
where $C$ is a constant only depending on $K$.
\end{lem}
\begin{proof}
We apply \cite[Th. 17.1.3]{cube-Ho} to the sixth order elliptic
operator $\Delta^3$ obtaining that, for any $r\in (0,1)$ and
$k=0,1,\dots, 6$, we have

\begin{equation}
\begin{aligned}
\label{2-A1-cube}
 \left\Vert d^k(x)D^{k}
u\right\Vert_{L^2(B_r)} \leq C\left(\left\Vert d^{6}(x)\Delta^3
u\right\Vert_{L^2(B_r)}+\left\Vert
u\right\Vert_{L^2(B_r)}\right)^{\frac{k}{6}}\left\Vert
u\right\Vert_{L^2(B_r)}^{1-\frac{k}{6}},
\end{aligned}
\end{equation}
where
$$d(x)=\mbox{dist}\left(x,\partial B_r\right)=r-|x|,
\quad x\in B_r$$ and $C>0$ is an absolute constant.

By applying Young inequality
$$a^{\beta}b^{1-\beta}\leq \beta \varepsilon
a+(1-\beta)\varepsilon^{-\frac{\beta}{1-\beta}}b,$$ for every
$a,b\geq 0$, $\beta\in[0,1)$, $\varepsilon>0$, and by using
\eqref{Cacciop-A1-1} and \eqref{2-A1-cube} we get

\begin{equation}
 \begin{aligned}\label{3-A1-cube}
\sum_{k=0}^5\int_{B_r}d^{2k}(x)\left\vert D^{k} u\right\vert^2
dx&\leq C\varepsilon^2\int_{B_r}d^{12}(x)|\Delta^3
u|^2dx+C_{\varepsilon}\left\Vert
u\right\Vert_{L^2(B_r)}^2\leq\\&\leq CK^2\varepsilon^2
\sum_{k=0}^5\int_{B_r}d^{12}(x)\left\vert D^{k} u\right\vert^2
dx+C_{\varepsilon}\left\Vert u\right\Vert_{L^2(B_r)}^2,
\end{aligned}
\end{equation}
where $C>0$ is an absolute constant and $C_{\varepsilon}>0$ depends on
$\varepsilon$ only. Hence we have

\begin{equation}\label{star-A2-cube}
\sum_{k=0}^5\int_{B_r}\left(1-CK^2\varepsilon^2d^{12-2k}(x)\right)d^{2k}(x)\left\vert
D^{k} u\right\vert^2\leq C_{\varepsilon}\left\Vert
u\right\Vert_{L^2(B_r)}^2.
\end{equation}
Now, if $\varepsilon=\left(\frac{1}{2CK^2}
\right)^{1/2}$ then
$$1-CK^2\varepsilon^2d^{12-2k}(x)\geq 1-CK^2\varepsilon^2\geq \frac{1}{2}, \quad
k=0,1,\dots , 5.$$ Hence we have

\begin{equation}\label{1-A3-cube}
\sum_{k=0}^5\left(\frac{r}{2}\right)^{2k}\int_{B_{r/2}}\left\vert
D^{k} u\right\vert^2 dx\leq\sum_{k=0}^5\int_{B_r}d^{2k}(x)\left\vert
D^{k} u\right\vert^2 dx\leq C\left\Vert u\right\Vert_{L^2(B_r)}^2.
\end{equation}
Furthermore, by \eqref{2-A1-cube} for $k=6$, \eqref{1-A3-cube} and
\eqref{Cacciop-A1-1} we get

\begin{eqnarray}\label{2-A3-cube}
&&\int_{B_r}d^{12}(x)\left\vert D^{6} u\right\vert^2 dx\leq
C\int_{B_r}d^{12}(x)|\Delta^3 u|^2dx+C\left\Vert
u\right\Vert_{L^2(B_r)}^2\leq \nonumber\\
&&C
\sum_{k=0}^5\int_{B_r}d^{12}(x)\left\vert D^{k} u\right\vert^2
dx+C\left\Vert u\right\Vert_{L^2(B_r)}^2\leq C\left\Vert
u\right\Vert_{L^2(B_r)}^2 , 
\end{eqnarray}
where $C>0$ depends on $K$ only.

Now, by \eqref{1-A3-cube} and \eqref{2-A3-cube} we have

\begin{equation}\label{2n-A3-cube}
\sum_{k=0}^6\left(\frac{r}{2}\right)^{2k}\int_{B_{r/2}}\left\vert
D^{k} u\right\vert^2 dx\leq\sum_{k=0}^6\int_{B_r}d^{2k}(x)\left\vert
D^{k} u\right\vert^2 dx\leq C\left\Vert u\right\Vert_{L^2(B_r)}^2
\end{equation}
and \eqref{eq:12a.2-cube} follows.
\end{proof}

\begin{theo} [{\bf Doubling inequality}]
    \label{theo:40.teo}
    Let $M$ be a positive number and $R_1$ the number introduced in
    Proposition \ref{prop:Carleman 3Laplace}. Assume that $U\in H^6\left(B_{1}\right)$ satisfy
\begin{equation}
    \label{2-1-appdoub}
    \left|\Delta^3 U\right|\leq M\left(\left|D\Delta^2 U\right|+\sum_{k=0}^4\left|D^k
    U\right|\right).
\end{equation}
There exists $C>1$, only depending on $M$, such that,
for every $r<\frac{R_1}{2^8}$ we have
\begin{equation}\label{eq:10.6.1102-cube}
    \int_{B_{2r}}U^2\leq C N^{{\overline{k}}}\int_{B_{r}}U^2
\end{equation}
where
\begin{equation}
    \label{eq:10.6.1108-cube}
    N=\frac{\int_{B_{R_1}}U^2}{\int_{B_{R_1/2^7}}U^2}
    \end{equation}
(with ${\overline{k}}=8$).
\end{theo}

\begin{lem}\label{new1-cube}
Let $U\in H^6\left(B_{1}\right)$ satisfy \eqref{2-1-appdoub}. Then there exists an absolute constant
$R_1\in(0,1]$ such that
for every $R$ and for every $r$ such that $0<2r<R<\frac{R_1}{2}$, we have
\begin{gather}
    \label{Lemma-doub-cube}
    R(2r)^{-2\tau}\int_{B_{2r}}U^2+R^{1-2\tau}\int_{B_{R} }
    U^2
    \leq C \overline{M}^2\left[\left(\frac{r}{2^7}\right)^{-2\tau}\int_{B_r}U^2+
    \left(\frac{R_1}{2^6}\right)^{-2\tau}\int_{B_{R_1}}U^2
    \right],
\end{gather}
for every $\tau\geq \widetilde{\tau}\geq 1$, with $\widetilde{\tau}$
depending on $M$ only and $C$ a positive absolute constant.
\end{lem}

\begin{proof}
Let $r,R$ satisfy
\begin{equation}
    \label{eq:25.1-cube}
0<2r<R<\frac{R_1}{2}.
\end{equation}
Let $\eta\in C^\infty_0((0,R_1))$ such that
\begin{equation}
    \label{eq:25.2-cube}
    0\leq \eta\leq 1,
\end{equation}
\begin{equation}
    \label{eq:25.4-cube}
\eta=0 \quad \hbox{ in }\left(0,\frac{r}{4}\right)\cup
\left(\frac{2}{3}R_1,1\right), \quad \eta=1 \quad \hbox{ in
}\left[\frac{r}{2}, \frac{R_1}{2}\right],
\end{equation}
\begin{equation}
    \label{eq:25.6-cube}
\left|\frac{d^k\eta}{dt^k}(t)\right|\leq C r^{-k} \quad \hbox{ in
}\left(\frac{r}{4}, \frac{r}{2}\right),\quad\hbox{ for } 0\leq k\leq
6,
\end{equation}
\begin{equation}
    \label{eq:25.7-cube}
\left|\frac{d^k\eta}{dt^k}(t)\right|\leq C R_1^{-k} \quad \hbox{ in
}\left(\frac{R_1}{2}, \frac{2R_1}{3}\right),\quad\hbox{ for } 0\leq
k\leq 6.
\end{equation}
Let us define
\begin{equation}
    \label{eq:25.5-cube}
\xi(x)=\eta(|x|).
\end{equation}
Let us fix $\epsilon=\frac{1}{5}$  and let us shift $\tau$ in
$\tau-4$ in estimate \eqref{eq:24.4-3L}, by adjusting the exponent
of such an estimate, we have

\begin{equation}
 \begin{aligned}
    \label{Lemma-doub:1-1}
    \tau^6r^3\int_{B_{R_1}}\rho^{-3-2\tau}u^2dx
+\tau\int_{B_{R_1}}\rho^{11-2\tau}\left|D\Delta^2u\right|^2+&\sum_{k=0}^4\tau^{9-2k}\int_{B_{R_1}}
\rho^{2k+1-2\tau}|D^ku|^2\leq\\& \leq
C\int_{B_{R_1}}\rho^{12-2\tau}|\Delta^3 u|^2,
\end{aligned}
\end{equation}

\noindent for every $\tau\geq \overline{\tau}$, for every $r\in (0,
R_1)$ and for every $u\in C^\infty_0(B_{R_1}\setminus\
\overline{B}_{r/4})$, where $R_1$ has been introduced in Proposition \ref{prop:Carleman 3Laplace} and is an absolute constant since here we have chosen $\epsilon=\frac{1}{5}$.

\medskip

By a density argument, we may apply the Carleman estimate
\eqref{Lemma-doub:1-1} to $u=\xi U$, obtaining
\begin{eqnarray}
    \label{eq:26.1-cube}
   && \tau^6r^3\int_{B_{R_1}}\rho^{-3-2\tau}\xi^2U^2dx+\tau\int_{B_{R_1}}\rho^{11-2\tau}\left|D\Delta^2(\xi U)\right|^2
    +\sum_{k=0}^4\tau^{9-2k}\int_{B_{R_1}}
\rho^{2k+1-2\tau}|D^k(\xi U)|^2 \leq\quad \quad \nonumber\\
&&\quad \quad \leq C\int_{B_{R_1}}\rho^{12-2\tau}|\Delta^3 (\xi U)|^2
\end{eqnarray} for
$\tau\geq \overline{\tau}$ and $C$ an absolute constant. Since we
have

\begin{equation}\label{Lemma-doub:2-2}
\left|\Delta^3 (\xi U)\right|^2\leq 2\xi^2\left|\Delta^3
U\right|^2+C\sum_{k=0}^5 \left|D^k U\right|^2 \left|D^{6-k}
\xi\right|^2,
\end{equation}
denoting

\begin{equation}
    \label{eq:27.1-cube}
    J_0 =\int_{B_{r/2}\setminus B_{r/4}}\rho^{12-2\tau}
    \sum_{k=0}^5 \left(r^{k-6}|D^k U|\right)^2,
\end{equation}
\begin{equation}
    \label{eq:27.2-cube}
    J_1 =\int_{B_{2R_1/3}\setminus B_{R_1/2}}\rho^{12-2\tau}
    \sum_{k=0}^5 (R_1^{k-6}|D^k U|)^2,
\end{equation}
we have
\begin{gather}\label{eq:27.3-cube}
\tau^6r^3\int_{B_{R_1}}\rho^{-3-2\tau}\xi^2U^2dx+\tau\int_{B_{R_1}}\rho^{11-2\tau}\left|D\Delta^2(\xi
U)\right|^2+
    +\sum_{k=0}^4\tau^{9-2k}\int_{B_{R_1}}
\rho^{2k+1-2\tau}|D^k(\xi U)|^2 \leq\\\nonumber
\leq C\int_{B_{R_1}}\rho^{12-2\tau}|\Delta^3 U|^2+CJ_0+CJ_1,
\end{gather}
for $\tau\geq \overline{\tau}$, with $C$ an absolute constant.

Now, by using \eqref{eq:25.1-cube}--\eqref{eq:25.5-cube},
\eqref{Lemma-doub:2-2}, performing a trivial estimate {}from below of
the left hand side of \eqref{eq:27.3-cube} and a trivial estimate
{}from above of the right hand side of \eqref{eq:27.3-cube}, we get

\begin{eqnarray}
    \label{eq:33.1-cube}
   && \tau^6r^3\int_{B_{R_1}}\rho^{-3-2\tau}|\xi U|^2+\tau\int_{B_{R_1/2} \setminus
B_{r/2}}\rho^{11-2\tau}\left|D\Delta^2 U\right|^2
    +\sum_{k=0}^4 \tau^{9-2k}\int_{B_{R_1/2} \setminus
B_{r/2}}
    \rho^{2k+1-2\tau}|D^k U|^2
    \leq \nonumber\\
   && \leq
    CM^2\int_{B_{R_1/2} \setminus
B_{r/2}}\rho^{12-2\tau}|D\Delta^2U|^2
    +CM^2\int_{B_{R_1/2} \setminus
B_{r/2}} \sum_{k=0}^4 \tau^{8-2k}\int_{B_{R_1/2} \setminus B_{r/2}}
    \rho^3\cdot\rho^{2k+1-2\tau}|D^k U|^2+ \nonumber\\ 
   &&\quad \quad \quad +C\overline{M}^2(J_0+J_1),
\end{eqnarray}
for $\tau\geq \overline{\tau}$, with $C$ an absolute constant and
$\overline{M}=\sqrt{M^2+1}$.

Let us move on the left of \eqref{eq:33.1-cube} the first and the
second term on the right of \eqref{eq:33.1-cube} and we obtain

\begin{gather}
    \label{Lemma-doub:6}
    \tau^6r^3\int_{B_{R_1}}\rho^{-3-2\tau}|\xi U|^2+\int_{B_{R_1/2} \setminus
B_{r/2}}\left(\tau-CM^2\rho\right) \rho^{11-2\tau}\left|D\Delta^2 U\right|^2+\\
\nonumber
    +\sum_{k=0}^4 \int_{B_{R_1/2} \setminus
B_{r/2}}\tau^{8-2k}\left(\tau-CM^2\rho^3\right) \rho^{2k+1-2\tau}
|D^kU|^2
    \leq C\overline{M}^2(J_0+J_1),
\end{gather}
for $\tau\geq \overline{\tau}$, where $C$ is the same constant that
appears in \eqref{eq:33.1-cube}.

Now, taking into account that $\rho\leq 1$ in $B_{R_1}$, by
\eqref{Lemma-doub:6} we obtain
\begin{eqnarray}
    \label{Lemma-doub:1-7}
   && \tau^6r^3\int_{B_{R_1}}\rho^{-3-2\tau}|\xi U|^2+\frac{\tau}{2}\int_{B_{R_1/2} \setminus
B_{r/2}} \rho^{11-2\tau}\left|D\Delta^2 U\right|^2
    +\frac{1}{2}\sum_{k=0}^4 \tau^{9-2k}\int_{B_{R_1/2} \setminus
B_{r/2}} \rho^{2k+1-2\tau} |D^kU|^2 \nonumber\\
    &&\quad \quad \leq C\overline{M}^2(J_0+J_1),
\end{eqnarray}
for $\tau\geq \widetilde{\tau}$, where (recall that
$\overline{\tau}\geq 1$)
$$\widetilde{\tau}=\max\left\{2C\overline{M}^2, \overline{\tau}\right\}.$$

 Let us estimate $J_0$ and $J_1$. 
 
 We start by observing that by \eqref{eq:stima_rho}  for any $x\in B_{r/2}\setminus B_{r/4}$ we have that 
 \begin{equation}\label{stimacoronacircolare}
 \frac{r}{2^7}\le \rho(x)\le \frac{r}{2} \ .
 \end{equation}
 
{}From \eqref{eq:12a.2-cube}, \eqref{eq:27.1-cube} and 
\eqref{stimacoronacircolare}, we have

\begin{equation}
 \begin{aligned}\label{0-9}
J_0&= \int_{B_{r/2}\setminus B_{r/4}}\rho^{12-2\tau}
    \sum_{k=0}^5 \left(r^{k-6}|D^k U|\right)^2\leq
    C\left(\frac{r}{2^7}\right)^{-2\tau}\sum_{k=0}^5\int_{B_{r/2}}\left(r^{k}|D^k
    U|\right)^2\leq
C\left(\frac{r}{2^7}\right)^{-2\tau}\int_{B_{r}}U^2 , \quad \quad \quad \quad 
\end{aligned}
\end{equation}
where $C$ depends on $M$.

Similarly, we observe that 
by \eqref{eq:stima_rho}  for any $x\in B_{2R_1/3}\setminus B_{R_1/2}$ we have that 
 \begin{equation}\label{stimacoronacircolare2}
 \frac{R_1}{2^6}\le \rho(x)\le \frac{2R_1}{3} \ .
 \end{equation}

Again {}from \eqref{eq:12a.2-cube}, \eqref{eq:27.2-cube} and 
\eqref{stimacoronacircolare2}, we have

\begin{equation}
    \label{eq:37.1-cube}
     J_1
     \leq C\left(\frac{R_1}{2^6}\right)^{-2\tau}\int_{B_{R_1}}U^2.
\end{equation}
By \eqref{Lemma-doub:1-7}, \eqref{0-9}, \eqref{eq:37.1-cube} we have

\begin{gather}
\label{Lemma-doub:2-9}
    \tau^6r^3\int_{B_{R_1}}\rho^{-3-2\tau}|\xi U|^2+\tau^{9}\int_{B_{R_1/2} \setminus
B_{r/2}} \rho^{1-2\tau} U^2  \leq
C\overline{M}^2\left(\left(\frac{r}{2^7}\right)^{-2\tau}\int_{B_{r}}U^2+\left(\frac{R_1}{2^6}\right)^{-2\tau}\int_{B_{R_1}}U^2\right),
\end{gather}
for every $\tau\geq \widetilde{\tau}$.

Now, recalling that $2r<R<\frac{R_1}{2}$, by \eqref{eq:25.2-cube}
\eqref{eq:25.4-cube}, we have trivially

\begin{equation}
\label{Lemma-doub:2n-9}
    \tau^6r^3\int_{B_{R_1}}\rho^{-3-2\tau}|\xi U|^2\geq \frac{1}{8}(2r)^{-2\tau}\int_{B_{2r}\setminus B_{r/2}} U^2,
\end{equation}
and
\begin{equation}
\label{Lemma-doub:2nn-9}
    \tau^{9}\int_{B_{R_1/2} \setminus
B_{r/2}} \rho^{1-2\tau} U^2\geq
\left(R\right)^{1-2\tau}\int_{B_{R}\setminus B_{r/2}} U^2.
\end{equation}

By \eqref{Lemma-doub:2-9}, \eqref{Lemma-doub:2n-9} and
\eqref{Lemma-doub:2nn-9} we have

\begin{gather}
    \label{10.6.9-cube}
    (2r)^{-2\tau}\int_{B_{2r}\setminus B_{r/2} }U^2+R^{1-2\tau}\int_{B_{R} \setminus B_{r/2}}
    U^2\leq C \overline{M}^2\left[\left(\frac{r}{2^7}\right)^{-2\tau}\int_{B_r}U^2+
    \left(\frac{R_1}{2^6}\right)^{-2\tau}\int_{B_{R_1}}U^2
    \right],
\end{gather}
for every $\tau\geq \widetilde{\tau}$. Now, adding
$2R(2r)^{-2\tau}\int_{B_{r/2} }U^2$ to both sides of
\eqref{10.6.9-cube} we get the wished estimate
\eqref{Lemma-doub-cube} for $r<R/2$ and $R<\frac{R_1}{2}$.
\end{proof}

\bigskip

\begin{proof}[Proof of Theorem \ref{theo:40.teo}.]

Let us fix $R=\frac{R_1}{2^7}$ in \eqref{Lemma-doub-cube} obtaining

\begin{gather}
    \label{eq:10.6.952-cube}
    \frac{R_1}{2^7} (2r)^{-2\tau}\int_{B_{2r}}U^2+\left(\frac{R_1}{2^7}\right)^{1-2\tau}\int_{B_{R_1/2^7}}
    U^2
    \leq C \overline{M}^2\left[\left(\frac{r}{2^7}\right)^{-2\tau}\int_{B_r}U^2+
    \left(\frac{R_1}{2^6}\right)^{-2\tau}\int_{B_{R_1}}U^2
    \right],
\end{gather}
for every $\tau\geq \widetilde{\tau}$, where $\widetilde{\tau}$
depends on $M$ only  and $C$ is an absolute constants.

Now, choosing $\tau=\tau_0$, where

\begin{equation}
   \label{eq:10.6.1005-cube}
    \tau_0=\widetilde{\tau}+\log_4\frac{2^7C \overline{M}^2N}{R_1}
\end{equation}
and

\begin{equation}
    \label{eq:10.6.1008-cube}
    N=\frac{\int_{B_{R_1}}U^2}{\int_{B_{R_1/2^7}}
    U^2}
\end{equation}
we have
$$\left(\frac{R_1}{2^7}\right)^{1-2\tau}\int_{B_{R_1/2^7}}
    U^2\geq
C
\overline{M}^2\left(\frac{R_1}{2^6}\right)^{-2\tau}\int_{B_{R_1}}U^2.$$
Hence, by \eqref{eq:10.6.952-cube}, we obtain
\begin{equation}
    \label{eq:10.6.1014-cube}
    \frac{R_1}{2^7}(2r)^{-2\tau_0}\int_{B_{2r}}U^2
    \leq C \overline{M}^2\left(\frac{r}{2^7}\right)^{-2\tau_0}\int_{B_r}U^2,
\end{equation}
where $C$ is an absolute constant. Using \eqref{eq:10.6.1005-cube}
and \eqref{eq:10.6.1014-cube}, we have

\begin{equation}
    \label{eq:10.6.1024-cube}
    \int_{B_{2r}}U^2
    \leq CN^{\overline{k}}\int_{B_r}U^2,
\end{equation}
where $C$ depends on $M$ only and $\overline{k}=8$.

The proof is complete.
\end{proof}

\bigskip

\begin{cor}[{\bf Doubling inequality and three sphere inequality}]\label{SUCP}
Assume that $U\in H^6\left(B_{1}\right)$ satisfies inequality
\eqref{2-1-appdoub}.

Then there exists an absolute constant $R_1\in(0,1]$ such that for every $r\leq s\leq\frac{R_1}{2^8}$, we have

\begin{equation}\label{SUCP-1}
    \int_{B_{s}}U^2\leq
    CN^{\overline{k}}\left(\frac{s}{r}\right)^{\log_2\left(CN^{\overline{k}}\right)} \int_{B_r} U^2,
\end{equation}
where $N$ is given by \eqref{eq:10.6.1108-cube}.

In addition, if $2r\leq s\leq\frac{R_1}{2^8}$ then we have

\begin{equation}\label{Tre-sfere-U}
    \int_{B_{s}}U^2\leq
    \left(C\int_{B_{R_1}}U^2\right)^{1-\widetilde{\theta}(s,r)}\left(\int_{B_r}U^2\right)^{\widetilde{\theta}(s,r)},
\end{equation}
where
\begin{equation*}
   \widetilde{\theta}(s,r)=\frac{1}{1+2\overline{k}\log_2\frac{s}{r}}.
\end{equation*}
\end{cor}

\begin{proof}

Let us prove \eqref{SUCP-1}. Let $r<s\leq\frac{R_1}{2^8}$. Denote
$$j=\left[\log_2\left(sr^{-1}\right)\right],$$ (where, for every
$a\in \mathbb{R}^+$, $[a]$ denotes the integer part of $a$). We have
$$2^jr\leq s <2^{j+1}r.$$  By iteration,
\eqref{eq:10.6.1102-cube} gives
\begin{equation*}
    \int_{B_{s}}U^2\leq \int_{B_{2^{j+1}r}}U^2
    \leq \left(CN^{\overline{k}}\right)^{j+1}\int_{B_r}U^2\leq CN^{\overline{k}}\left(\frac{s}{r}\right)^{\log_2(CN^{\overline{k}})}\int_{B_r}U^2
\end{equation*}
and \eqref{SUCP-1} follows.
Now let us prove \eqref{Tre-sfere-U}. By elementary properties of
logarithm function and by \eqref{SUCP-1} we have, for $2r\leq s\leq
\frac{R_1}{2^8}$,
\begin{equation}\label{SUCP-8}
    \int_{B_{s}}U^2\leq
    \left(CN^{\overline{k}}\right)^{2\log_2\frac{s}{r}}\int_{B_r}U^2,
\end{equation}
Now, by \eqref{eq:10.6.1108-cube} we have trivially
\begin{equation*}
    N=\frac{\int_{B_{R_1}}U^2}{\int_{B_{\frac{R_1}{2^7}}}U^2}\leq
    \frac{\int_{B_{R_1}}U^2}{\int_{B_{s}}U^2}.
\end{equation*}
By the last inequality and by \eqref{SUCP-8} we have
\begin{equation*}
    \left(\int_{B_{s}}U^2\right)^{1+2\overline{k}{\log}_2\frac{s}{r}}\leq
    \left(C\int_{B_{R_1}}U^2\right)^{2\overline{k}\log_2\frac{s}{r}}\int_{B_r}U^2
\end{equation*}
which implies \eqref{Tre-sfere-U}.

\end{proof}

\section{Appendix}\label{Appendix}

\begin{proof}[Proof of Lemma \ref{Lemma:Fichera}]

Formula \eqref{eq:change_var_first} is standard.
By a density argument it is not restrictive to assume $w\in C^2(\overline{\Omega})$.
Let $s$ be an arclenght defined on $\partial\Omega$. Locally $\partial\Omega$ is represented either as $(\xi_1,g(\xi_1))$, or as $(g(\xi_2), \xi_2)$, with $g\in C^2[a,b]$. 
To fix ideas, let us assume that $\partial\Omega$ is locally represented as $(\xi_1,g(\xi_1))$. The arclength $s$, up to an additive constant, is given by
\begin{equation*}
   s(\xi_1)=\pm\int_a^{\xi_1}\sqrt{1+(g'(t))^2} dt
\end{equation*}
the sign depending on the fact that the domain $\Omega$ is described locally either as
$\{x=(x_1,x_2)\ |\ x_2>g(x_1)\}$ or as $\{x=(x_1,x_2)\ |\ x_2<g(x_1)\}$ respectively.

To fix ideas, let us consider the first situation, so that we have
\begin{equation}
   \label{eq:s'}
   s'(\xi_1)=\sqrt{1+(g'(\xi_1))^2}, \quad \xi_1'(s)=\frac{1}{\sqrt{1+(g'(\xi_1(s)))^2}}.
\end{equation}
The unit tangent and outer normal vector at $(\xi_1,g(\xi_1))$
are given by
\begin{equation}
   \label{eq:vettore tangente}
   \tau=\left(\frac{1}{\sqrt{1+(g'(\xi_1))^2}}, \frac{g'(\xi_1)}{\sqrt{1+(g'(\xi_1))^2}},
	\right),
\end{equation}
\begin{equation}
   \label{eq:vettore normale}
   n=\left(\frac{g'(\xi_1)}{\sqrt{1+(g'(\xi_1))^2}}, \frac{-1}{\sqrt{1+(g'(\xi_1))^2}}
	\right).
\end{equation}
It is useful to notice that $n_1=\tau_2$, $n_2=-\tau_1$.

We have
\begin{equation}
   \label{eq:6-1}
   \tau_1,_s=-n_2,_s=\frac{-g'g''}{(1+(g')^2)^2}|_{\xi_1(s)}, \quad \tau_2,_s=n_1,_s=\frac{g''}{(1+(g')^2)^2}|_{\xi_1(s)}.
\end{equation}
Recalling that the curvature ${\cal K}$ is given by
\begin{equation}
   \label{eq:3bis-4}
   {\cal K}=\frac{g''}{(1+(g')^2)^{3/2}}|_{\xi_1(s)},
\end{equation}
we can rewrite the above formulas in the following form
\begin{equation}
   \label{eq:3ter-4}
   n_\alpha,_s={\cal K}\tau_\alpha, \quad \ \tau_\alpha,_s=-{\cal K}n_\alpha.
\end{equation}

Let us introduce a local coordinate system $(y_1,y_2)=(s,z)$, where $s$ is the arclength parameter and $z$ is the shift along the direction of the outer unit normal $n$.
Let us consider the following map
\begin{multline*}
   x=\varphi(y),\quad (x_1,x_2)=\varphi(y_1,y_2) =\varphi(s,z)\qquad s\in I, 
	z\in(-\delta,\delta),\\
	\varphi(s,z)=(\xi_1(s), g(\xi_1(s)))+zn=\left(\xi_1(s)+z\frac{g'(\xi_1(s))}{\sqrt{1+(g'(\xi_1(s)))^2}},
	g(\xi_1(s))-z\frac{1}{\sqrt{1+(g'(\xi_1(s)))^2}}
	\right).
\end{multline*}
where $I$ is an open interval and $\delta$ is a positive constant.

The Jacobian matrix of $\varphi$ is
\begin{equation}
   \label{eq:2-2}
   J\varphi(s,z)=
	\left(
  \begin{array}{cc}
    \frac{1}{\sqrt{1+(g')^2}}+\frac{zg''}{(1+(g')^2)^2}\ \  & \frac{g'}{\sqrt{1+(g')^2}}\\
		 &  \\
		g'\left(\frac{1}{\sqrt{1+(g')^2}}+\frac{zg''}{(1+(g')^2)^2}\right)\ \ &  \frac{-1}{\sqrt{1+(g')^2}}
  \end{array}
\right),
\end{equation}
where $g'$ and $g''$ have to be computed in $\xi_1(s)$.

The determinant of the above matrix is given by
\begin{equation*}
   det(J\varphi)=-1-\frac{zg''}{(1+(g')^2)^{3/2}}.
\end{equation*}
Notice that $det(J\varphi)\neq 0$ in a suitable neighboorhood of
	$I\times \{0\}$, whose image through the map $\varphi$ is a neighboorhood of a portion of $\partial\Omega$.
	
Therefore $\varphi$ is locally invertible, and let $f$ its inverse:
\begin{equation*}
   y=f(x), \quad (s,z)=(y_1,y_2)=f(x_1,x_2)=(f_1(x_1,x_2), f_2(x_1,x_2)).
\end{equation*}

The Jacobian matrix of $f$ is
\begin{equation}
   \label{eq:2-3}
   Jf(x_1,x_2)=
	\left(
  \begin{array}{cc}
    \left(\sqrt{1+(g')^2}+\frac{zg''}{1+(g')^2}\right)^{-1}\ \  & g'\left(\sqrt{1+(g')^2}+\frac{zg''}{1+(g')^2}\right)^{-1}\\
		 &  \\
		\frac{g'}{\sqrt{1+(g')^2}}\ \ &  \frac{-1}{\sqrt{1+(g')^2}}
  \end{array}
\right),
\end{equation}
where $g'$ and $g''$ have to be computed in $\xi_1(f_1(x_1,x_2))$ and $z$ has to be computed in 
$f_2(x_1,x_2)$.

Let us recall the following formula (see for instance \cite[p. 458, formula (7.101)]{PF})
\begin{eqnarray}
   \label{eq:1-4}
   w,_{\alpha\beta}(x)=&&\sum_{i,j=1}^2w,_{y_iy_j}(y)f_i,_{x_\beta}(x)f_i,_{x_\alpha}(x)+
	\sum_{i,j=1}^2 w,_{y_i}(y)f_j,_{x_\alpha}(x)(f_i,_{x_\beta}),_{y_j}(y)= \nonumber \\
	&&= I+II,
\end{eqnarray}
where $y=f(x)$. Replacing $y_1$ and $y_2$ with $s$ and $z$ respectively and noticing that $(\cdot),_n=(\cdot),_z$ at boundary points, we have that
\begin{equation}
   \label{eq:2-4}
   I=w,_{ss}\tau_\alpha\tau_\beta+w,_{nn}n_\alpha n_\beta+
w,_{sn}(\tau_\alpha n_\beta+\tau_\beta n_\alpha),
\end{equation}
\begin{equation}
   \label{eq:3-4}
   II=w,_{s}\left(\tau_\alpha (f_1,_{x_\beta}),_{s}+n_\alpha (f_1,_{x_\beta}),_{z}
	\right)+
 w,_{n}\left(\tau_\alpha (f_2,_{x_\beta}),_{s}+n_\alpha (f_2,_{x_\beta}),_{z}\right).
\end{equation}

Recalling \eqref{eq:6-1} and \eqref{eq:3ter-4}, 
we can compute at boundary points ($z=0$)
\begin{equation}
   \label{eq:1-5} f
   (f_1,_{x_\beta}),_{s}=\tau_\beta,_s=-{\cal K}n_\beta,  \quad (f_1,_{x_\beta}),_{z}=-n_\beta,_s=-{\cal K}\tau_\beta,
\end{equation}
\begin{equation}  
   \label{eq:3-5}
   (f_2,_{x_\beta}),_{s}=n_\beta,_s={\cal K}\tau_\beta, \quad  (f_2,_{x_\beta}),_{z}=0.
\end{equation}
Therefore, we have
\begin{eqnarray}
   \label{eq:4bis-5}
   &&II=w,_s(\tau_\alpha\tau_\beta,_s-n_\alpha n_\beta,_s)+w,_n\tau_\alpha n_\beta,_s=w,_s(-{\cal K}\tau_\alpha n_\beta-{\cal K}n_\alpha \tau_\beta)+w,_n{\cal K}\tau_\alpha \tau_\beta=\nonumber\\
	&&= w,_s(\tau_\beta\tau_\alpha,_s-n_\beta n_\alpha,_s)+w,_n\tau_\beta n_\alpha,_s
\end{eqnarray}
and {}from \eqref{eq:2-4} and the above formula we have \eqref{eq:change_var_secondBIS}.
If $\partial\Omega$ is locally represented as $(g(\xi_2),\xi_2)$, by inverting the role of the variables, we directly get \eqref{eq:change_var_secondBIS}.
\end{proof}

\section*{Acknowledgments}
The work of ER, ES and SV was performed under the \text{PRIN grant No. 201758MTR2-007}. ES and SV have also been supported by \text{Gruppo Nazionale per l'Analisi Matematica, la Probabilità e le loro applicazioni (GNAMPA)} by the grant "Problemi inversi per equazioni alle derivate parziali". 

\bibliographystyle{plain}

\end{document}